\documentclass[UTF-8,reqno]{amsart}
\usepackage{enumerate}
\usepackage{verbatim,listings}
\usepackage{amssymb,url,color, booktabs}

\usepackage{mathrsfs}
\usepackage{amsmath}
\usepackage{verbatim}
\usepackage{fancyhdr}
\usepackage[nobysame]{amsrefs}
\BibSpec{article}{%
+{}{\PrintAuthors} {author}
+{,}{ \textrm} {title}
+{.}{ \textit} {journal}
+{,}{ \textbf} {volume}
+{}{ \parenthesize} {date}
+{,}{ } {pages}
+{.}{ arXiv:} {eprint}
+{.}{} {transition}
}
\BibSpec{book}{%
+{}{\PrintAuthors} {author}
+{,}{ \textit} {title}
+{.}{ \textrm} {series} 
+{,}{ Vol.} {volume}
+{.}{ } {publisher}
+{,}{ } {date}
+{.}{} {transition}
}
\usepackage{color}
\usepackage{tikz}
\usepackage{graphicx}
\usepackage[colorlinks=true]{hyperref}
\hypersetup{
    linkcolor=blue,          
    citecolor=red,        
    filecolor=blue,      
    urlcolor=cyan
}

\numberwithin{equation}{section}

\newcommand{\be}{\begin{eqnarray}}
\newcommand{\ee}{\end{eqnarray}}
\newcommand{\ce}{\begin{eqnarray*}}
\newcommand{\de}{\end{eqnarray*}}
\newtheorem{theorem}{Theorem}[section]

\newtheorem{remark}{Remark}[section]

\newtheorem{lemma}[theorem]{Lemma}
\newtheorem{definition}[theorem]{Definition}
\newtheorem{proposition}[theorem]{Proposition}
\newtheorem{Examples}[theorem]{Example}
\newtheorem{corollary}[theorem]{Corollary}

\newenvironment{proof of theorem 1.4}{{\it Proof of Theorem 1.4}.}{{\hfill 	
		$\square$\hskip - \parfillskip}}
\newenvironment{proof of theorem 1.5}{{\it Proof of Theorem 1.5}.}{{\hfill 	
		$\square$\hskip - \parfillskip}}
\newenvironment{proof of theorem 1.6}{{\it Proof of Theorem 1.6}.}{{\hfill 	
		$\square$\hskip - \parfillskip}}

\makeatletter

\newcommand{\Rmnum}[1]{\expandafter\@slowromancap\romannumeral #1@}
\makeatother

\usepackage{caption}
\def\p{\partial}

\def\g{\gamma}

\def\[{{\Big[}}
\def\]{{\Big]}}
\def\<{{\langle}}
\def\>{{\rangle}}
\def\({{\Big(}}
\def\){{\Big)}}

\def\bx{{\mathbf{x}}}

\def\={&\!\!=\!\!&}

\def\1{{\mathbf{1}}}

\def\geq{\geqslant}
\def\leq{\leqslant}

\def\le{\leqslant}

\def\k{\kappa}

\def\p{\partial}

\def\g{\gamma}

\def\[{{\Big[}}
\def\]{{\Big]}}
\def\<{{\langle}}
\def\>{{\rangle}}
\def\({{\Big(}}
\def\){{\Big)}}

\def\bx{{\mathbf{x}}}

\def\W{{\mathcal W}}
\def\C{\mathcal{C}}
\def\K{\mathcal{K}}

\def\={&\!\!=\!\!&}
\def\bt{\begin{theorem}}
\def\et{\end{theorem}}
\def\bl{\begin{lemma}}
\def\el{\end{lemma}}
\def\br{\begin{remark}}
\def\er{\end{remark}}
\def\bx{\begin{Examples}}
\def\ex{\end{Examples}}
\def\bd{\begin{definition}}
\def\ed{\end{definition}}
\def\bp{\begin{proposition}}
\def\ep{\end{proposition}}
\def\bc{\begin{corollary}}
\def\ec{\end{corollary}}

\def\geq{\geqslant}
\def\leq{\leqslant}

\def\le{\leqslant}

 \def\R{\mathbb R}
 \def\R{\mathbb R}    
  \def\m{{\bf m}}
   
\def\<{\langle} \def\>{\rangle}

\def\bpf{\begin{proof}}
\def\epf{\end{proof}}

\allowdisplaybreaks

	\tikzset{
	pattern size/.store in=\mcSize,
	pattern size = 5pt,
	pattern thickness/.store in=\mcThickness,
	pattern thickness = 0.3pt,
	pattern radius/.store in=\mcRadius,
	pattern radius = 1pt}
\makeatletter
\pgfutil@ifundefined{pgf@pattern@name@_ak5ydppfi}{
	\pgfdeclarepatternformonly[\mcThickness,\mcSize]{_ak5ydppfi}
	{\pgfqpoint{0pt}{0pt}}
	{\pgfpoint{\mcSize+\mcThickness}{\mcSize+\mcThickness}}
	{\pgfpoint{\mcSize}{\mcSize}}
	{
		\pgfsetcolor{\tikz@pattern@color}
		\pgfsetlinewidth{\mcThickness}
		\pgfpathmoveto{\pgfqpoint{0pt}{0pt}}
		\pgfpathlineto{\pgfpoint{\mcSize+\mcThickness}{\mcSize+\mcThickness}}
		\pgfusepath{stroke}
}}
\makeatother
\tikzset{every picture/.style={line width=0.75pt}} 

\begin{document}
\title{Alexandrov-Fenchel inequalities for convex anisotropic capillary hypersurfaces in the half-space}\thanks{\it {The research is partially supported by NSFC (Nos. 11871053 and 12261105).}}
\author{Jinyu Gao and Guanghan Li}

\thanks{Email address: 
	 jinyugao@whu.edu.cn; ghli@whu.edu.cn}

\address{Department of Mathematics,
	Zhejiang Sci-Tech University,
	Hangzhou 310018, China}
\address{School of Mathematics and Statistics, Wuhan University, Wuhan 430072, China.
}

\begin{abstract}
 In this paper, the results of Mei, Wang, Weng and Xia \cite{Xia-arxiv} on capillary convex bodies are extended to the anisotropic setting.
 We develop a theory for anisotropic capillary convex bodies in the half-space and establish a general Alexandrov-Fenchel inequality for mixed volumes of anisotropic capillary convex bodies. Thus, this weakens the conditions of the inequality in \cite{Ding-Gao-Li-arxiv}*{Theorem 1.4} and extends it to a more general case.
\end{abstract}

\maketitle
\setcounter{tocdepth}{2}
\tableofcontents

\section{Introduction}
In  \cite{Young}, Young  formulates the equilibrium condition for the contact angle of a capillarity surface commonly known as Young’s law. We know that Euclidean capillary convex bodies in the upper half-space characterize the equilibrium states of liquid droplets contacting a hyperplane, while anisotropic capillary convex bodies may describe stable configurations of droplets with non-uniform density adhering to the hyperplane.
For comprehensive treatments of isotropic or anisotropic capillarity, we refer to \cite{Finn,Philippis}.

We consider the properly, embedded, smooth compact  hypersurface $\Sigma$ in the closure $\overline{\mathbb{R}^{n+1}_+}$ of half-space ${\mathbb{R}_+^{n+1}}=\{x\in \mathbb{R}^{n+1} :\<x,E_{n+1}\>> 0 \}$ with $\mathrm{int}(\Sigma)\subset \mathbb{R}_+^{n+1}$ and $\partial\Sigma\subset\partial\overline{\mathbb{R}_+^{n+1}}$, where $E_{n+1}$ denotes the $(n+1)$-coordinate unit vector. $\Sigma$
 is called a $\theta$-capillary hypersurface if $\Sigma$ intersects $\partial\overline{\mathbb{R}^{n+1}_+}$ at a constant contact angle $\theta \in (0,\pi)$. If $\Sigma$ is an anisotropic capillary hypersurface (as defined in \cite{Jia-Wang-Xia-Zhang2023} or below), the contact angle on the boundary is typically non-constant.

 To describe the anisotropic capillary hypersurface, we let $\mathcal{W}\subset \mathbb{R}^{n+1}$ be a given smooth closed strictly convex hypersurface containing the origin $O$ with a support function $F:\mathbb{S}^{n}\rightarrow\mathbb{R}_+$. The Cahn-Hoffman map associated with $F$ is given by
 \begin{align}\label{equ:Psi}
 \Psi:\mathbb{S}^n\rightarrow\mathbb{R}^{n+1},\quad	\Psi(x)=F(x)x+\nabla^{\mathbb{S}} {F}(x),
 \end{align}
where $\nabla^{\mathbb{S}}$ denotes the covariant derivative on $\mathbb{S}^n$ with standard round metric. $\mathcal{W}$ is called a Wulff shape with support function $F$. Let $\Sigma$ be a smooth compact orientable hypersurface in $\overline{\mathbb{R}_+^{n+1}}$ with boundary $\partial\Sigma\subset\partial\overline{\mathbb{R}_+^{n+1}}$, which encloses a bounded domain $\widehat{\Sigma}$. Let $\nu$ be the unit normal  of $\Sigma$ pointing outward of  $\widehat{\Sigma}$. Given a constant  $\omega_0 \in(-F(E_{n+1}), F(-E_{n+1}))$, we say $\Sigma$ is an \textbf{anisotropic $\omega_0$-capillary hypersurface} (see \cite{Jia-Wang-Xia-Zhang2023}) if
\begin{align}
	\label{equ:w0-capillary}
	\<\Psi(\nu),-E_{n+1}\>=\omega_0,\quad \text{on}\ \partial\Sigma.
\end{align}
If additionally, $\Sigma$ is convex, we call $\widehat{\Sigma}$ the \textbf{anisotropic $\omega_0$-capillary convex body} associated with $\Sigma$.

For the convenience of description, we need a constant vector $E_{n+1}^F \in \mathbb{R}^{n+1}$ (ref. \cite{Jia-Wang-Xia-Zhang2023,Xia-arxiv}) defined as
 $$
 E_{n+1}^F= \begin{cases}\frac{\Psi\left(E_{n+1}\right)}{F\left(E_{n+1}\right)}, & \text { if } \omega_0<0, \\[5pt]
 	E_{n+1},  & \text { if } \omega_0=0,
 	\\[5pt]
  -\frac{\Psi\left(-E_{n+1}\right)}{F\left(-E_{n+1}\right)}, & \text { if } \omega_0>0.\end{cases}
 $$
 Note that $E_{n+1}^F$ is the unique vector in the direction $\Psi\left(E_{n+1}\right)$, whose scalar product with $E_{n+1}$ is 1 (see \cite{Jia-Wang-Xia-Zhang2023}*{eq. (3.2)}).

 The model example of anisotropic $\omega_0$-capillary hypersurface is the $\omega_0$-capillary Wulff shape $\W_{r_0,\omega_0}$ (see \cite{Jia-Wang-Xia-Zhang2023}) defined by \eqref{equ:crwo}, which is a part of a Wulff shape in $\overline{\mathbb{R}^{n+1}_+}$ such that the anisotropic capillary boundary condition \eqref{equ:w0-capillary} holds. 
 We denote $\mathcal{C}_{\omega_0}:=\W_{1,\omega_0}=\mathcal{T}(\W\cap\{x\in\mathbb{R}^{n+1}:\<x,E_{n+1}\>\geq -\omega_0\})$, where
 \begin{align}
 	\mathcal{T}:x\rightarrow x+\omega_0E_{n+1}^F,\quad
 \text{for}\  x\in \mathbb{R}^{n+1} ,
\label{equ:T=}
 \end{align}
  is a translation transformation.

 Let $(\Sigma,g)$ be an  anisotropic $\omega_0$-capillary hypersurface in $\overline{\mathbb{R}^{n+1}_+}$ with boundary supported on $\partial\overline{\mathbb{R}^{n+1}_+}$, where metric $g$ is induced from Euclidean space $(\mathbb{R}^{n+1},\<\cdot,\cdot\>)$. In \cite{Ding-Gao-Li-arxiv}, we introduced the anisotropic capillary quermassintegrals as follows:
 \begin{align*}
 \mathcal{V}_{0,\omega_0}(\Sigma):=|\widehat{\Sigma}|,
 \quad	\mathcal{V}_{1,\omega_0}({\Sigma})=\frac{1}{n+1}\left(
 	|\Sigma|_F+\omega_0|\widehat{\partial\Sigma}|
 	\right),
 \end{align*}
 and for $1\leq k\leq n$,
 \begin{align*}
 	\mathcal{V}_{k+1,\omega_0}(\Sigma)=
 	\frac{1}{n+1}
 	\left(
 	\int_{\Sigma}H^F_kF(\nu)d\mu_g
 	+\frac{\omega_0}{n}
 	\int_{\partial\Sigma} H_{k-1}^{\bar{F}}\bar{F}(\bar{\nu})ds
 	\right),
 \end{align*}
 where $\widehat{\partial\Sigma}=\widehat{\Sigma}\cap \partial\overline{\mathbb{R}^{n+1}_+}$ denotes the domain enclosed  by $\partial\Sigma\subset\mathbb{R}^n=\partial\overline{\mathbb{R}^{n+1}_+}$,
 $|\cdot |$ denotes the volume of the domain in $ \mathbb{R}^n$ or $ \mathbb{R}^{n+1}$,
 $|\cdot |_F$ denotes the anisotropic area defined by \eqref{equ:F-area},
 $\bar{\nu}$  is the unit outward co-normal  of $\partial \Sigma\subset \mathbb{R}^n=\partial\overline{\mathbb{R}^{n+1}_+}$,
 $ds$ is the $(n-1)$-dimensional induced volume form of $\partial \Sigma$,
 $d\mu_g$ is the induced volume form of the metric $g$,
$\bar{F}$ is the support function of the new $(n-1)$-dimensional Wulff shape  $\overline{\W}=\partial\mathcal{C}_{\omega_0}$ in $\partial\overline{\mathbb{R}^{n+1}_+}=\mathbb{R}^n$,
$H_k^{{F}}(0\leq k\leq n,\,H_0^{{F}}=1,H^F_{n+1}=0)$ is the  normalized anisotropic $k$-mean curvature of $\Sigma\subset\mathbb{R}^{n+1}$ with respect to Wulff shape ${\W}$,
and $H_k^{\bar{F}}(0\leq k\leq n-1,\,H_0^{\bar{F}}=1)$ is the  normalized anisotropic $k$-mean curvature of $\partial\Sigma\subset\mathbb{R}^n$ with respect to Wulff shape $\overline{\W}$.

Specifically, if we take $\W$ to be the unit sphere, then $F = 1$, and the anisotropic case reduces to the classical isotropic geometry, and  the anisotropic capillary quermassintegrals
 matches the definition of  isotropic quermassintegrals \cite{Wang-Weng-Xia}*{eq.(1.8)} for  capillary hypersurfaces in the half-space.

 As shown in \cite{Ding-Gao-Li-arxiv}*{Theorem 6.2}, the first variation of $\mathcal{V}_{k+1,\omega_0}$ gives rise to $H^F_{k+1}$ with $-1\leq k\leq n$. Then in \cite{Ding-Gao-Li-arxiv}*{Theorem 1.4}, for convex anisotropic $\omega_0$-capillary hypersurface $\Sigma$, we  used the flow method to prove the following Alexandrov-Fenchel inequality for $l=0,1\leq k\leq n-1$ under certain conditions  (\cite{Ding-Gao-Li-arxiv}*{Condition 1.1}) on the given  Wulff shape $\W$ and constant $\omega_0$,
 \begin{align}\label{equ:A-F-lk}
 	\left(\frac{\mathcal{V}_{k, \omega_0}({\Sigma})}{|\widehat{\mathcal{C}}_{\omega_0}|}\right) ^{\frac{1}{n+1-k}}
 	\geq
 	\left(\frac{\mathcal{V}_{l, {\omega_0}}({\Sigma})}{|\widehat{\mathcal{C}}_{\omega_0}|}\right)^{\frac{1}{n+1-l}}, \quad 0\leq l \leq k\leq  n,
 \end{align}
 with equality holding if and only if $\Sigma$ is an $\omega_0$-capillary Wulff shape. Here $\widehat{\mathcal{C}}_{\omega_0}$  is the anisotropic $\omega_0$-capillary convex body  determined by $\mathcal{C}_{\omega_0}$. In \cite{Ding-Gao-Li-arxiv}*{Remark 1.1 (2)}, we remark that, if $\W=\mathbb{S}^n$ (i.e., the isotropic setting), the Condition 1.1 in \cite{Ding-Gao-Li-arxiv} is equivalent to constant contact angle $\theta\in \left(0,\frac{\pi}{2}\right]$.

 It is natural to consider whether the Condition 1.1 in \cite{Ding-Gao-Li-arxiv} can be removed and whether a more general range of $k$ and $l$, beyond $l=0,1\leq k\leq n-1$, can be addressed.

 In \cite{Xia-arxiv}, Mei, Wang, Weng and Xia proved the more general  Alexandrov-Fenchel inequality for mixed volumes of capillary convex bodies
 \begin{align}\label{equ:V>vv}
 	V^2\left(K_1, K_2, K_3, \cdots, K_{n+1}\right) \geq V\left(K_1, K_1, K_3, \cdots, K_{n+1}\right) V\left(K_2, K_2, K_3, \cdots, K_{n+1}\right),
 \end{align}
 where $V\left(K_1, K_2, \cdots, K_{n+1}\right)$ is the so-called mixed volume of
  a family of $\theta$-capillary convex bodies $K_i \subset \overline{\mathbb{R}_+^{n+1}}, i=1, \cdots, n+1$.
  And it is demonstrated that  the capillary quermassintegrals have alternative integral expressions which match the definition of mixed volumes. Furthermore, the Alexandrov-Fenchel inequality for quermassintegrals of $\theta$-capillary hypersurface with $\theta\in(0,\pi)$ is derived, and thus  this removes the restriction condition $\theta\in(0,\frac{\pi}{2}]$ under which the inequality was proven by flow method in \cite{Wang-Weng-Xia}.

  We remark that for the case of classical Euclidean convex body geometry in $\mathbb{R}^{n+1}$, 
  if we replace the Euclidean geometric quantities in the integral form of the mixed volume $V(K_1,\cdots,K_{n+1})$ with anisotropic geometric quantities, the value of the mixed volume remains unchanged according to \eqref{equ:Q(AB)=Q(A)det(B)}, \eqref{equ:S_F}, and \eqref{u-hatu}.
   Therefore, there is no necessity to ``generalize" the inequality \eqref{equ:V>vv} from classical closed hypersurfaces in $\mathbb{R}^{n+1}$ (Euclidean framework) to the anisotropic case for closed hypersurfaces in $\mathbb{R}^{n+1}$. However, the situation differs for capillary convex bodies with boundaries: the Gauss map image of an anisotropic capillary convex hypersurface typically does not form a spherical cap but corresponds to irregular regions on the sphere. When transforming the mixed volume into an integral over spherical regions, only special cases can be reduced to the Euclidean capillary boundary scenario. Since the  inequality \eqref{equ:V>vv} for capillary convex bodies  cannot encompass the anisotropic capillary case, it becomes essential to extend inequality \eqref{equ:V>vv} for capillary convex bodies to the anisotropic capillary convex body framework.

 Inspired by \cite{Xia-arxiv}, we also need to  establish a theory for anisotropic capillary convex bodies in half-space. However, due to the asymmetry of the anisotropic Weingarten matrix \eqref{equ:S_F} and the complexity of anisotropic geometry, some challenges arise when extending Proposition 2.9 and Lemma 2.14 (then Proposition 2.15) in \cite{Xia-arxiv}.

Fortunately, in \cite{Ding-Gao-Li-arxiv}*{Lemma 7.4 (i)}, we established an alternative integral expression for the anisotropic capillary quermassintegrals $\mathcal{V}_{k+1,\omega_0}$ for $0\leq k\leq n$:
 \begin{align*}
 	\mathcal{V}_{k+1,\omega_0}(\Sigma)=\frac{1}{n+1}\int_{\Sigma}H_k^F\(F(\nu)+\omega_0\<\nu,E_{n+1}^F\>\) d\mu_g,
 \end{align*}
 then it can be proved that the  anisotropic capillary quermassintegrals match the mixed volume for anisotropic capillary convex bodies (see Lemma \ref{lemma:V-k=V(KKK,LLL)}), which is an extension of \cite{Xia-arxiv}*{Proposition 2.15}.
 Therefore, the main challenge of this paper lies in proving Lemma \ref{lemma:V12=V21}.

Similar to \cite{Xia-arxiv}, we define the anisotropic capillary Gauss map by $\widetilde{\nu_{{F}}}=\mathcal{T}\nu_{{F}}=\nu_{{F}}+\omega_0E_{n+1}^F$ and we use it to reparametrize anisotropic  support function $\hat{u}$ (see \eqref{u-hatu}) by $\hat{s}=\hat{ u}\circ \widetilde{\nu_{{F}}}^{-1}$ as a function on $\mathcal{C}_{\omega_0}$, which we called the \textbf{anisotropic capillary support function}, where $\nu_{{F}}$ and $\hat{ u}$ are respectively the anisotropic Gauss map and anisotropic support function defined in Section \ref{sec 2}. 
 Let  $\mathcal{K}_{\omega_0}$ denote the collection of all anisotropic $\omega_0$-capillary convex bodies. We prove the following anisotropic capillary version of Alexandrov-Fenchel inequality with rigidity characterization, which extends Theorem 1.1 in \cite{Xia-arxiv}.

\begin{theorem}\label{thm:A-F-convex}
Given a Wulff shape $\W$ (or a Minkowski norm $F$) and a constant
  $\omega_0 \in\left(-F\left(E_{n+1}\right), F\left(-E_{n+1}\right)\right)$. Let $\widehat{\Sigma}_i \in \mathcal{K}_{\omega_0}$ for $1 \leq i \leq n+1$. Then
  \begin{align*}
  	V^2\left(\widehat{\Sigma}_1, \widehat{\Sigma}_2, \widehat{\Sigma}_3, \cdots, \widehat{\Sigma}_{n+1}\right) \geq V\left(\widehat{\Sigma}_1, \widehat{\Sigma}_1, \widehat{\Sigma}_3, \cdots, \widehat{\Sigma}_{n+1}\right) V\left(\widehat{\Sigma}_2, \widehat{\Sigma}_2, \widehat{\Sigma}_3, \cdots, \widehat{\Sigma}_{n+1}\right).
  \end{align*}
  Equality holds if and only if the anisotropic capillary support functions $\hat{s}_j$ of $\widehat{\Sigma}_j, j=1,2$, satisfy
  \begin{align}\label{equ:s=as+...}
  	\hat{s}_1=a \hat{s}_2+\sum_{i=1}^n a_i\frac{\<\nu, E_i\>}{F(\nu)}, \quad \text { on } \quad \mathcal{C}_{\omega_0},
  \end{align}
  for some constants $a>0, a_i \in \mathbb{R}, i=1, \cdots, n$, and $\left\{E_i\right\}_{i=1}^n$ the horizontal coordinate unit vectors of $\overline{\mathbb{R}_{+}^{n+1}}$. 
 Here $\nu(\xi)$ is the unit  normal of $\xi\in\mathcal{C}_{\omega_0}$ pointing out of $\widehat{\mathcal{C}}_{\omega_0}$. 
\end{theorem}
 We remark that, for any $\widehat\Sigma\in\mathcal{K}_{\omega_0}$, it holds that $\widetilde{\nu_F}(\Sigma)=\mathcal{C}_{\omega_0}$ by Lemma \ref{Lemma:gaussmap}.  
  Denote by $\xi=\widetilde{\nu_F}(X)$ for $X\in\Sigma$, it is easy to check that $T_X\Sigma=T_{\xi}\mathcal{C}_{\omega_0}$, then the unit outward normals  of $\Sigma$ and $\mathcal{C}_{\omega_0}$ coincide at the corresponding points. Therefore, we use $\nu $ to denote the unit outer normal of $\Sigma$ or $\mathcal{C}_{\omega_0}$, depending on the context. Similarly, the unit outer co-normal vectors (of $\partial\Sigma$ in $\Sigma$) $\mu$ are treated in the same manner.

As a consequence of Theorem \ref{thm:A-F-convex}, we derive:
\begin{theorem}\label{thm:A-F-k}
	Given a Wulff shape $\W$ (or a Minkowski norm $F$) and a constant
	$\omega_0 \in\left(-F\left(E_{n+1}\right), F\left(-E_{n+1}\right)\right)$.
	For any anisotropic $\omega_0$-capillary convex body $\widehat{\Sigma} \in \mathcal{K}_{\omega_0}$, \eqref{equ:A-F-lk} is true, with equality holding if and only if $\Sigma$ is an $\omega_0$-capillary Wulff shape.
\end{theorem}
\begin{remark}
	If $\W=\mathbb{S}^n$, $\Sigma$ is a $\theta$-capillary hypersurface with $\omega_0=-\cos\theta$, and $\omega_0 \in\left(-F\left(E_{n+1}\right), F\left(-E_{n+1}\right)\right)$ is equivalent to  $\theta\in\left(0,\pi\right)$. Therefore, Theorem \ref{thm:A-F-convex} and \ref{thm:A-F-k} can be viewed as an extension of results in \cite{Xia-arxiv}.
\end{remark}

We mention that, there are numerous studies on the Alexandrov-Fenchel  inequality, as detailed in \cite{book-convex-body,Xia-arxiv,Shenfeld23,McCoy05,Guan-Li-09,Xia2017,Wei-Xiong-22,Wei-Xiong-21,Wang-Weng-Xia,Wang-Weng-Xia-2024-inverseMC,Mei-Wang-Weng-2024-MCF,Weng-Xia-22-ball,Hu-Wei-Yang-Zhou,Makuicheng,Ding-Li-JFA,Gao-Li-JGA}, and references therein.

The rest of this paper is organized as follows. In Section \ref{sec 2}, we briefly introduce some preliminaries on the anisotropic capillary hypersurfaces. In Section \ref{sec 3}, some properties about capillary convex bodies are obtained, and especially, several key lemmas are proved.
Section \ref{sec 4} proves the main Theorem \ref{thm:A-F-convex} and \ref{thm:A-F-k}.

\section{Preliminary}\label{sec 2}
\subsection{The Wulff shape and dual Minkowski norm}
Let $F$ be a smooth positive function on the standard sphere $(\mathbb{S}^n, g^{\mathbb{S}^n} ,\nabla^{\mathbb{S}})$ such that the matrix
\begin{equation*}
A_{F}(x)~=~\nabla^{\mathbb{S}}\nabla^{\mathbb{S}} {F}(x)+F(x)g^{\mathbb{S}^n}, 
\quad x\in \mathbb{S}^n,
\end{equation*}
is positive definite on $  \mathbb{S}^n$,
where 
 $g^{\mathbb{S}^n}$ denotes the round metric on $\mathbb{S}^n$.
  Then there exists a unique smooth strictly convex hypersurface $\W$ defined by
\begin{align*}
\W=\{\Psi(x)|\Psi(x):=F(x)x+\nabla^{\mathbb{S}} {F}(x),~x\in \mathbb{S}^n\},
\end{align*}
whose support function is given by $F$ (see \cite{Xia2017}). We call $\W$ the Wulff shape determined by the function $F\in C^{\infty}(\mathbb{S}^n)$. When $F$ is a constant, the Wulff shape is just a round sphere.

The smooth function $F$ on $\mathbb{S}^n$ can be extended  to a $1$-homogeneous function on $\mathbb{R}^{n+1}$ by
\begin{equation*}
F(x)=|x|F({x}/{|x|}), \quad x\in \mathbb{R}^{n+1}\setminus\{0\},
\end{equation*}
and setting $F(0)=0$. Then it is easy to show that $\Psi(x)=DF(x)$ for $x\in \mathbb{S}^n$, where $D$ denotes the standard gradient on $\mathbb{R}^{n+1}$.


The homogeneous extension $F$ defines a Minkowski norm on $\mathbb{R}^{n+1}$, that is, $F$ is a norm on $\mathbb{R}^{n+1}$ and $D^2(F^2)$ is uniformly positive definite on $\mathbb{R}^{n+1}\setminus\{0\}$ (see \cite{Xia-phd}). We can define a dual Minkowski norm $F^0$ on $\mathbb{R}^{n+1}$ by
\begin{align*}
F^0(\xi):=\sup_{x\neq 0}\frac{\langle x,\xi\rangle}{{F}(x)},\quad \xi\in \mathbb{R}^{n+1}.
\end{align*}
We call $\W$ the unit Wulff shape since
 $$\W=\{x\in \R^{n+1}: F^0(x)=1\}.$$
A Wulff shape of radius $r_0$ centered at $x_0$ is given by
\begin{align*}
	\W_{r_0}(x_0)=\{x\in\mathbb{R}^{n+1}:F^0(x-x_0)=r_0\}.
\end{align*}
An $\omega_0$-capillary Wulff shape of radius $r_0$ (see \cite{Jia-Wang-Xia-Zhang2023}) is given by
\begin{align}\label{equ:crwo}
	\W_{r_0,\omega_0}(E):=\{x\in\overline{\mathbb{R}_+^{n+1}}:F^0(x-r_0\omega_0E)=r_0\},
\end{align}
 which is a part of a Wulff shape cut by a hyperplane  $\{x_{n+1}=0\}$, here $E$ should satisfy $\<E,E_{n+1}\>=1$ which guarantees that  $\W_{r_0,\omega_0}(E)$ satisfies anisotropic capillary condition \eqref{equ:w0-capillary}. It is easy to see that $E-E_{n+1}^F\in \partial\overline{\mathbb{R}^{n+1}_+}$, since $\<E_{n+1},E_{n+1}^F\>=1$. If there is no confusion, we often take $E=E_{n+1}^F$ and we just write $\W_{r_0,\omega_0}:=\W_{r_0,\omega_0}(E^F_{n+1})$ in this paper.

\subsection{Anisotropic curvature}
Let $(\Sigma,g,\nabla)\subset \overline{\mathbb{R}^{n+1}_+}$ be a $C^2$ hypersurface with $\partial\Sigma\subset\partial\overline{\mathbb{R}^{n+1}_+}$. 
The anisotropic Gauss map of $\Sigma$  is defined by $$\begin{array}{lll}\nu_F: &&\Sigma\to  \W\\
&&X\mapsto \Psi(\nu(X))=F(\nu(X))\nu(X)+\nabla^\mathbb{S} F(\nu(X)).\end{array} $$
The anisotropic principal curvatures $\k^F=(\k^F_1,\cdots, \k^F_n)$ of $\Sigma$ with respect to $\W$ at $X\in \Sigma$  are defined as the eigenvalues of
 \begin{align}\label{equ:S_F}
S_F=\mathrm{d}\nu_F=\mathrm{d}(\Psi\circ\nu)=A_F\circ \mathrm{d}\nu : T_X \Sigma\to T_{\nu_F(X)} \W=T_X \Sigma.
\end{align}

We define the normalized $k$-th elementary symmetric function $H^F_k=\sigma_k(\kappa^F)/\binom{n}{k}$ of the anisotropic principal curvature $\kappa^F$:
\begin{align*}
H^F_k:=\binom{n}{k}^{-1}\sum_{1\leq {i_1}<\cdots<{i_k}\leq n} \kappa^F_{i_1}\cdots \kappa^F_{i_k},\quad k=1,\cdots,n,
\end{align*}
where $\binom{n}{k}=\frac{n!}{k!(n-k)!}$.
Setting $H^F_{0}=1$, $H^F_{n+1}=0$, and $H_F=nH_1^F$
for convenience. To simplify the notation, from now on, we use  $\sigma_k$ to denote function $\sigma_k(\kappa^F)$ of the anisotropic principal curvatures, if there is no confusion.

\subsection{New metric and anisotropic formulas}\label{subsec 2.3}
There are new metric on $\Sigma$ and $\mathbb{R}^{n+1}$, which were introduced by Andrews in \cite{And01} and reformulated by Xia in \cite{Xia13}.  This new Riemannian metric $G$ with respect to $F^0$ on $\mathbb{R}^{n+1}$ is defined as
\begin{align*}
	G(\xi)(V,W):=\sum_{\alpha,\beta=1}^{n+1}\frac{\partial^2 \frac12(F^0)^2(\xi)}{\partial \xi^\alpha\partial \xi^\beta} V^\alpha W^\beta, \quad\hbox{ for } \xi\in \mathbb{R}^{n+1}\setminus \{0\}, V,W\in T_\xi{\mathbb{R}^{n+1}}.
\end{align*}
The third order derivative of $F^0$ gives a $(0,3)$-tensor
\begin{equation}\nonumber
	Q(\xi)(U,V,W):=\sum_{\alpha,\beta,\gamma=1}^{n+1} Q_{\alpha\beta\gamma}(\xi)U^\alpha V^\beta W^\gamma:=\sum_{\alpha,\beta,\gamma=1}^{n+1} \frac{\partial^3(\frac12(F^0)^2(\xi)}{\partial \xi^\alpha \partial \xi^\beta \partial \xi^\gamma}U^\alpha V^\beta W^\gamma,
\end{equation}
for $\xi\in \mathbb{R}^{n+1}\setminus \{0\},$ $U,V,W\in T_\xi{\mathbb{R}^{n+1}}.$

When we restrict the metric $G$ to $\mathcal{W}$,  the $1$-homogeneity of $F^0$ implies that
\begin{eqnarray*}
	&G(\xi)(\xi,\xi)=1,  \quad G(\xi)(\xi, V)=0, \quad \hbox{ for } \xi\in \W,\  V\in T_\xi \W.
	\\&Q(\xi)(\xi, V, W)=0, \quad \hbox{ for } \xi\in\W,\  V, W\in \mathbb{R}^{n+1}.
\end{eqnarray*}
For a smooth hypersurface $\Sigma$ in $\overline{\mathbb{R}_+^{n+1}}$, since $\nu_F(X)\in \W$ for $X\in \Sigma$, we have
\begin{align}
	&G(\nu_F)(\nu_F,\nu_F)=1, \quad G(\nu_F)(\nu_F, V)=0, \quad \hbox{ for } V\in T_X \Sigma,\nonumber
	\\
	&Q(\nu_F)(\nu_F, V, W)=0, \quad \hbox{ for } V, W\in \mathbb{R}^{n+1}.\nonumber
\end{align}	
This means  $\nu_F(X)$ is perpendicular to $T_X \Sigma$ with respect to the metric $G(\nu_F)$.
Then the Riemannian metric $\hat{g}$ on $\Sigma$ induced from $(\mathbb{R}^{n+1}, G)$  can be defined as
\begin{eqnarray*}
	\hat{g}(X):=G(\nu_F(X))|_{T_X \Sigma}, \quad X\in \Sigma.
\end{eqnarray*}

Denote by $\hat{g}_{ij}$ and $\hat{h}_{ij}$ the first and second fundamental form of $(\Sigma, \hat{g})\subset (\mathbb{R}^{n+1}, G)$, respectively, that is
$$\hat{g}_{ij}=G(\nu_F(X))(\p_i X, \p_j X),\quad \hat{h}_{ij}=-G(\nu_F(X))(\nu_F, \partial_i\p_j X).$$
If $\{\hat{g}^{ij}\}$ the inverse matrix of $\{\hat{g}_{ij}\}$, we can reformulate $\k^F$ as  the  eigenvalues of $\{\hat{h}_j^i\}=\{\hat{g}^{ik}\hat{h}_{kj}\}$.
It is easy to see that if $\Sigma=\W$, we have $\nu_F(\W)=X(\W)$, $\hat{h}_{ij}=\hat{g}_{ij}$ and $\kappa^F=(1,\cdots,1)$. If  $X\in\Sigma=\mathcal{C}_{\omega_0}\subset\mathcal{T}\W$, we have $\nu_F(X)=\mathcal{T}^{-1}X=X-\omega_0E_{n+1}^F$, then  $\kappa^F=(1,\cdots,1)$. In this case we denote $\mathring{g}:=G(\mathcal{T}^{-1}{\xi})|_{T_{{\xi}}\mathcal{C}_{\omega_0}}$ (for ${\xi}\in\mathcal{C}_{\omega_0}$), which is a Riemmannian metric on $\mathcal{C}_{\omega_0}$, and denote by $\mathring{\nabla}$ the Levi-Civita connection of $\mathring{g}$.

We denote by $\hat{\nabla}$ the Levi-Civita connections of $\hat{g}$ on $\Sigma$, then the anisotropic Gauss-Weingarten type formula and the anisotropic Gauss-Codazzi type equation are as follows.
\begin{lemma}[\cite{Xia13}*{Lemma 2.5}] \label{lem2-1}
	\begin{eqnarray}
		\partial_i\partial_j X=-\hat{h}_{ij}\nu_F+\hat{\nabla}_{\partial_i } \partial_j+\hat{g}^{kl}A_{ijl}\partial_kX; \;\;\;\hbox{ (Gauss formula)}\label{equ:Gauss-formula}
		\end{eqnarray}
		\begin{eqnarray*}
		\partial_i \nu_F=\hat{g}^{jk}\hat{h}_{ij}\partial_k X;\;\; \hbox{ (Weingarten formula) }
	\end{eqnarray*}
		 $$\hat{R}_{i j k \ell}=\hat{h}_{i k} \hat{h}_{j \ell}-\hat{h}_{i \ell} \hat{h}_{j k}+\hat{\nabla}_{\partial_{\ell}} A_{j k i}-\hat{\nabla}_{\partial_k} A_{j \ell i}+A_{j k}^m A_{m \ell i}-A_{j \ell}^m A_{m k i} ;\quad	\text{(Gauss equation)}$$
	\begin{eqnarray*}
		\hat{\nabla}_k\hat{h}_{ij}+\hat{h}_j^lA_{lki}=\hat{\nabla}_j\hat{h}_{ik}+\hat{h}_k^lA_{lji}. \;\;\hbox{ (Codazzi equation) }
	\end{eqnarray*}
	Here, $\hat{R}$ is the Riemannian curvature tensor of $\hat{g}$ and  $A$ is a $3$-tensor
	\begin{eqnarray*}
	A_{ijk}=-\frac12\left(\hat{h}_i^l Q_{jkl}+\hat{h}_j^l Q_{ilk}-\hat{h}_k^l Q_{ijl}\right),
	\end{eqnarray*} where $Q_{ijk}=Q(\nu_F)(\partial_i X, \partial_j X, \partial_k X)$. Note that the $3$-tensor $A$ on $(\Sigma,\hat{g})$ depends on $\hat{h}_i^j$. It is known that $Q$ is totally symmetric in all three indices, while $A$ is only symmetric for the first two indices.
\end{lemma}
Here $X$ and $\nu_F$ can be regarded as vector-valued functions in $\mathbb{R}^{n+1}$ with a fixed Cartesian coordinate. Hence, the terms $\partial_i\partial_j X$ and $\partial_i \nu_F$ are understood as the usual partial derivative on vector-valued functions. 

The case $\Sigma=\mathcal{C}_{\omega_0}$ is the most important case which will be used in this paper. Analogue to \cite{Xia-2017-convex}*{Proposition 2.2}, we rewrite Lemma \ref{lem2-1} in this case.

\begin{lemma}\label{lemma:Gauss-Weingarten}
	 Let $X \in \mathcal{C}_{\omega_0}\subset\mathcal{T}(\W)=\W+\omega_0E_{n+1}^F$.  We  have $\nu_{{F}}=\mathcal{T}^{-1}(X)$, $\partial_iX=\partial_i(\mathcal{T}^{-1}(X))$, $\mathring{h}_{ij}(X)=\mathring{g}_{ij}(X)=\hat{g}_{ij}(\mathcal{T}^{-1}(X))=\hat{h}_{ij}(\mathcal{T}^{-1}(X))$ with $\{e_i=\partial_iX\}_{i=1}^n\in T_X{\mathcal{C}_{\omega_0}}=T_{\mathcal{T}^{-1}(X)}\W$.
$$
A_{i j k}=-\frac{1}{2} Q_{i j k}=-\frac{1}{2} Q(\mathcal{T}^{-1}(X))\left(\partial_i X, \partial_j X, \partial_k X\right),
$$
and
\begin{align*}
	\mathring{\nabla}_i Q_{j k l}=\mathring{\nabla}_j Q_{i k l} .
\end{align*}
The Gauss-Weingarten formula and the Gauss equation are as follows:
\begin{align}
	\partial_i \partial_j X & =-\mathring{g}_{i j} \mathcal{T}^{-1}(X)+\mathring{\nabla}_{\partial_i} \partial_j-\frac{1}{2} \mathring{g}^{k l} Q_{i j l} \partial_k X ; \text { (Gauss formula) } \nonumber\\
	\partial_i \nu_F & =\partial_i X ; \text { (Weingarten formula) }\nonumber \\
	\mathring{R}_{i j k l} & =\mathring{g}_{i k} \mathring{g}_{j l}-\mathring{g}_{i l} \mathring{g}_{j k}+\frac{1}{4} \mathring{g}^{p m} Q_{j k p} Q_{m l i}-\frac{1}{4} \mathring{g}^{p m} Q_{j l p} Q_{m k i}. \text { (Gauss equation) }\nonumber
\end{align}
\end{lemma}
\subsection{Anisotropic support function and  area element}
Let $u=\< X,\nu \> $ be the support function of $X$, and $\hat{u}:=G(\nu_F)(\nu_F,X)$ be the anisotropic support function of $X$ with respect to $F$(or $\W$). Then we have (see e.g. \cite{Xia2017})
\begin{equation}
	\label{u-hatu}
	\hat{u}=\frac{u}{F(\nu)}.
\end{equation}
Similar to \eqref{u-hatu}, since  $DF^0(DF(x))=\frac{x}{F(x)}$ (see e.g. \cite{Xia-phd}), we obtain
\begin{align}\label{equ:G(vF,Y)=<v,Y>/F}
	G(\nu_F)(\nu_F,Y)=\<DF^0(DF(\nu)),Y\>=\<\frac{\nu}{F(\nu)},Y\>=\frac{\<Y,\nu\>}{F(\nu)},
\end{align}
for any $Y\in \mathbb{R}^{n+1}$.

Let us define the anisotropic area element of $\Sigma$ as
\begin{align}\label{equ:duF}
	{\rm d}\mu_F:=F(\nu){\rm d}\mu_g.
\end{align}

The anisotropic area of $\Sigma$ is given by
\begin{align}
	\label{equ:F-area}
	|\Sigma |_F=\int_{\Sigma} {\rm d}\mu_F.
\end{align}


There holds the following anisotropic Minkowski integral formula for $\omega_0$-capillary hypersurface.
\begin{lemma}[Jia-Wang-Xia-Zhang, 2023, \cite{Jia-Wang-Xia-Zhang2023}*{Theorem 1.3}] \label{Thm1.2}
	Let  $\Sigma \subset \overline{\mathbb{R}_{+}^{n+1}}$ be a $C^2$ anisotropic $\omega_0$-capillary hypersurface, where $\omega_0 \in\left(-F\left(E_{n+1}\right)\right.,$ $\left.F\left(-E_{n+1}\right)\right)$.
	 Then for  $0 \leq k \leq n-1$, it holds
	\begin{align*}
		\int_{\Sigma} H_{k}^F\left(1+\omega_0G(\nu_F)( \nu_F, E_{n+1}^F)\right)-H_{k+1}^F\hat{ u} \mathrm{~d}\mu_F=0.
	\end{align*}
	
\end{lemma}


\section{Anisotropic $\omega_0$-capillary convex bodies and mixed volume in $\overline{\mathbb{R}^{n+1}_+}$}\label{sec 3}
In this section, we shall study  anisotropic $\omega_0$-capillary convex bodies in $\overline{\mathbb{R}^{n+1}_+}$, in the frame work of  Mei, Wang, Weng and Xia \cite{Xia-arxiv} for Euclidean (isotropic) $\theta$-capillary convex bodies. 

\subsection{Anisotropic capillary convex bodies and convex functions}

	Given a Wulff shape $\W$ (or a Minkowski norm $F$). For $\omega_0\in\left(-F(E_{n+1}),F(-E_{n+1})\right)$, let $\Sigma$ be a strictly convex anisotropic $\omega_0$-capillary hypersurface in $\overline{\mathbb{R}_+^{n+1}}$ (satisfies anisotropic capillary boundary condition \eqref{equ:w0-capillary}), and $\widehat{\Sigma}\in\mathcal{K}_{\omega_0}$  the anisotropic $\omega_0$-capillary convex body associated with $\Sigma$. 

	The simplest capillary convex body $\widehat{\mathcal{C}}_{\omega_0}\in \K_{\omega_0}$ is the bounded domain enclosed by $\mathcal{C}_{\omega_0}$ and $\partial\overline{\mathbb{R}_+^{n+1}}$.
	
	Let $\mathcal{T}:\mathbb{R}^{n+1}\to \mathbb{R}^{n+1}$ be the translation given by \eqref{equ:T=}. 
	Instead of using the usual anisotropic Gauss map $\nu_F$, it is more convenient to use the following map
	\begin{eqnarray*}
		\widetilde {\nu_F}:= \mathcal{T} \circ \nu_F: \Sigma \to \mathcal{C}_{\omega_0},\end{eqnarray*}
	which we call \textbf{ anisotropic capillary Gauss map} of $\Sigma$.
	
	 Next for any $\widehat{\Sigma}\in \mathcal {K}_{\omega_0}$ we give a parametrization of $\Sigma$   via the anisotropic capillary Gauss map $\widetilde {\nu_F}$.
	\begin{lemma}\label{Lemma:gaussmap}
			Given a Wulff shape $\W$ (or a Minkowski norm $F$), and a constant  $\omega_0\in\left(-F(E_{n+1}),F(-E_{n+1})\right)$. Denote  $\mathcal{S}:=\{x\in\mathbb{S}^n:\<DF(x),E_{n+1}\>\geq -\omega_0\}$.  
			 For any  $\widehat{\Sigma}\in\mathcal{K}_{\omega_0}$, we have
		\begin{itemize}
			\item [(1)]  The Gauss map $\nu$ of $\Sigma$ has its spherical image $\mathcal{S}\subset\mathbb{S}^n$, and $\nu:\Sigma \to \mathcal{S}$ is a diffeomorphism.
			
			\item [(2)] The anisotropic capillary Gauss map $\widetilde{\nu_{{F}}}=\nu_{{F}}+\omega_0E_{n+1}^F$ of $\Sigma$ has its image in $\mathcal{C}_{\omega_0}=\(\W+\omega_0E_{n+1}^F\)\bigcap\overline{\mathbb{R}^{n+1}_+}$, and  $\widetilde {\nu_F}: \Sigma\to \mathcal{C}_{\omega_0}$  
	is a diffeomorphism.
		\end{itemize}
	\end{lemma}
	
	\begin{proof} (1) Since $\widehat{\Sigma}\in\mathcal{K}_{\omega_0}$, 
		following the approach established in  \cite[Corollary 2.4 (ii)]{Wang-Weng-Xia}, we obtain the relation $B_{\alpha\beta}(X)=\sin\theta(X)\cdot B^{\partial}_{\alpha\beta}(X)$ for $X\in\partial\Sigma$, where $\theta(X)$ is a function on $\partial\Sigma$ which satisfies $\<\nu(X),E_{n+1}\>=\cos\theta(X)$,  $e_{\alpha},e_{\beta}\in T_X(\partial\Sigma)$, $B$ and $B^{\partial}$ are the second fundamental form of $\Sigma\subset\overline{\R^{n+1}_+}$ and $\partial\Sigma \subset \mathbb{R}^n$ respectively. An argument parallel to that of \cite[Corollary 2.5]{Wang-Weng-Xia} shows that $\partial\Sigma\subset \partial\overline{\R^{n+1}_{+}}$ is a  strictly convex, closed  hypersurface. 
		
		 By \cite[Corollary 3.1 and Notes 3.2]{Ghomi}, we see that {the Gauss map $\nu$ of $\Sigma$ restricted to   $\p \Sigma$   is one-to-one} and furthermore, $\nu: \Sigma\to \nu(\Sigma)\subset \mathbb{S}^{n}$ is a diffeomorphism.
		 On the other hand,  it follows from \eqref{equ:w0-capillary} that $\nu(\partial \Sigma)=\partial\mathcal{S}$.
		  Also, since $\<DF(E_{n+1}),E_{n+1}\>=F(E_{n+1})>-\omega_0$, we have $E_{n+1}\in \mathcal{S}$ which lies in $\nu(\Sigma)$. It follows that $\nu(\Sigma)=\mathcal{S}$, then the assertion follows.
		
		  (2) Since  $\widehat{\mathcal{C}}_{\omega_0}\in\mathcal{K}_{\omega_0}$, by (1), the Gauss map $\nu^{\mathcal{C}_{\omega_0}}$ of $\mathcal{C}_{\omega_0}$ also has its image $\mathcal{S}$, and $\mathcal{C}_{\omega_0}=\mathcal{T}\(\W\cap\{x\in\mathbb{R}^{n+1}:\<x,E_{n+1}\>\geq-\omega_0\}\)\subset\mathcal{T}\W$ is diffeomorphic to its spherical image $\mathcal{S}$. This implies that $\(\nu^{\mathcal{C}_{\omega_0}}\)^{-1}\circ\nu: \Sigma\to\mathcal{C}_{\omega_0}$ is a diffeomorphic.
		
		 Denote by $\nu^{\W}$ the Gauss map of $\W$. Since $\nu^{\mathcal{C}_{\omega_0}}(\xi)=\nu^{\W}(\mathcal{T}^{-1}(\xi))$ for $\xi\in\mathcal{C}_{\omega_0}$ and $\nu_{{F}}(X)=DF(\nu(X))=\(\nu^{\W}\)^{-1}\circ\nu(X)$ for $X\in \Sigma$,  we have $\widetilde{\nu_{{F}}}=\mathcal{T}\circ\nu_{{F}}=\(\nu^{\mathcal{C}_{\omega_0}}\)^{-1}\circ\nu$. This completes the proof.
	\end{proof}

	From Lemma \ref{Lemma:gaussmap}, we can parametrize $\Sigma$ by the inverse anisotropic capillary Gauss map, i.e., $X:\mathcal{C}_{\omega_0}\to \Sigma$,  given by
	\begin{align}\label{equ:X(xi)=t-1VF}
		X(\xi)= \widetilde {\nu_F} ^{-1} (\xi) =\nu_F^{-1}\circ \mathcal{T}^{-1}(\xi)=\nu_F^{-1}(\xi-\omega_0E_{n+1}^F),
	\end{align}
	then $\nu_{{F}}(X(\xi))=\mathcal{T}^{-1}(\xi)$. The anisotropic support function of $\Sigma$ is given by
	\begin{eqnarray*}
		\hat{u}(X)=\frac{\<X,\nu(X)\>}{F(\nu(X))}=G(\nu_{{F}}(X))(X,\nu_{{F}}(X)).
	\end{eqnarray*}
	By the above parametrization, $\hat{s}(\xi)=\hat{u}(\widetilde{\nu_{{F}}}^{-1}(\xi))$ can be regarded as a function on $\xi\in\mathcal{C}_{\omega_0}$,
	\begin{align}\label{support}
		\hat{s}(\xi)
		 =&G(\mathcal{T}^{-1}(\xi))\(X(\xi),\mathcal{T}^{-1}(\xi)\)\nonumber
		\\=&G(\xi-\omega_0E_{n+1}^F)\(\widetilde{\nu_{{F}}}^{-1}(\xi),\xi-\omega_0E_{n+1}^F\)\nonumber
		\\=&\sup_{x\in\widehat{\Sigma}}G(\mathcal{T}^{-1}(\xi))\(\mathcal{T}^{-1}(\xi),x\).
	\end{align}
	{To distinguish with the anisotropic support function $\hat{u}$ of  ${\Sigma}$, we call $\hat{s}$ in \eqref{support} the \textbf{ anisotropic capillary support function} of $\Sigma$ (or of $\widehat{\Sigma}$).
		
		\begin{remark} In \cite{Ding-Gao-Li-arxiv,Arxiv}, $\bar{u}(X(\xi)):=\frac{\hat{s}(\xi)}{1+\omega_0G(\xi-\omega_0E_{n+1}^F)\(E_{n+1}^F,\xi-\omega_0E_{n+1}^F\)}$ for $\xi\in \mathcal{C}_{\omega_0}$, is called the anisotropic capillary support function of $\Sigma$, which has some properties such as  $\mathring{\nabla}_{\mu_F} \bar{u}=0$ on $\p\mathcal{C}_{\omega_0}$ with $\mu_{{F}}=A_F(\nu)\mu$, where  $\mu(X(\xi))$ is the unit outward co-normal of $X(\xi)\in\partial\Sigma\subset\Sigma$.
			So $\bar{u}$ is more suitable for this name.
		\end{remark}

		It is clear that  the anisotropic capillary Gauss map for $ \mathcal{C}_{\omega_0}$  is the identity map
		from $\mathcal{C}_{\omega_0} \to \mathcal{C}_{\omega_0}$ and	anisotropic capillary support function of $\mathcal{C}_{\omega_0}$ (or of $\widehat{\mathcal{C}}_{\omega_0}$) is
		\begin{align*}
			\hat{s}(\xi)=&G(\xi-\omega_0E_{n+1}^F)\(\xi  , \xi -\omega_0E_{n+1}^F\)
			\\
			=&
			G(\xi-\omega_0E_{n+1}^F)\(\xi-\omega_0E_{n+1}^F+\omega_0E_{n+1}^F  , \xi -\omega_0E_{n+1}^F\)
			\\
			=&1+\omega_0G(\xi-\omega_0E_{n+1}^F)\(E_{n+1}^F  , \xi -\omega_0E_{n+1}^F\).
		\end{align*}
		For simplicity, in the latter context, we denote
		\begin{eqnarray}\label{equ:s0}
			\hat{s}_o(\xi):=1+\omega_0G(\xi-\omega_0E_{n+1}^F)\(E_{n+1}^F  , \xi -\omega_0E_{n+1}^F\).
		\end{eqnarray} 

		Similar to \cite{Xia-arxiv}, we have the  following lemma for anisotropic setting.
		\begin{lemma}\label{lemma:Xia-Lemma2.4}
			For the parametrization $X:\left(\mathcal{C}_{\omega_0},\mathring{g},\mathring{\nabla }\right)\rightarrow\Sigma$ of anisotropic $\omega_0$-capillary convex body $\widehat{\Sigma}$, we have
			\begin{itemize}
			\item[(1)] $X(\xi)=\mathring{\nabla} \hat{s}(\xi) +\hat{s}(\xi)\mathcal{T}^{-1}(\xi)$.
				
			\item[(2)]  $\mathring{\nabla}_{\mu_F}\hat{s}=\frac{\omega_0}{F(\nu)\<\mu,E_{n+1}\>}\hat{s}$ along $\partial\mathcal{C}_{\omega_0}$. 
				
			\item[(3)] The anisotropic principal curvature radii of $\Sigma$ at $X(\xi)$ are given by the eigenvalues of $\tau_{ij}(\xi):=\mathring{\nabla}_{e_i}\mathring{\nabla}_{e_j}\hat{s}(\xi)+\mathring{g}_{i j} \hat{s}(\xi)-\frac{1}{2} Q_{i j k} \mathring{\nabla}_{e_k} \hat{s}(\xi)$  with respect to the metric $\mathring{g}$, where $\{e_i\}_{i=1}^n$ is an orthonormal frame on $(\mathcal{C}_{\omega_0},\mathring{g})$.  In particular, $(\tau_{ij})>0$ on $\mathcal{C}_{\omega_0}$.
		\end{itemize}	
	\end{lemma}
	
	\begin{proof}
	 (1) From \eqref{support}, for $\xi\in\mathcal{C}_{\omega_0}$, we have
	 		\begin{align}
				\mathring{\nabla}_{e_i} \hat{s}(\xi)=&G(\mathcal{T}^{-1}(\xi))\(D_{e_i} X, \xi -\omega_0E_{n+1}^F\)+G(\mathcal{T}^{-1}(\xi))\(X, e_i\)\nonumber
				\\
				&+Q(\mathcal{T}^{-1}(\xi))\(D_{e_i}(\mathcal{T}^{-1}(\xi)),\mathcal{T}^{-1}(\xi),X\)\nonumber
				\\
				=&G(\mathcal{T}^{-1}(\xi))\(X, e_i\), \quad 1\leq i\leq n.
				\label{equ:Ds}
			\end{align}
			Here the second equality follows from the fact that $Q(z)(z,\cdot,\cdot)=0$ and the fact that $D_{e_i} X$ is tangential to $\Sigma$ under the anisotropic metric $G(\mathcal{T}^{-1}(\xi))(\cdot,\cdot)$ in $T_{\xi}\mathbb{R}^{n+1}$, while $\mathcal{T}^{-1}(\xi)=\xi -\omega_0E_{n+1}^F=\nu_{{F}}(X(\xi))$ is anisotropic normal to $\Sigma$. The first assertion follows.
			
		(2) By the anisotropic capillary boundary condition, we have
		\begin{align*}
			-\omega_0=\<\nu_{{F}}(X(\xi)),E_{n+1}\>=\<\mathcal{T}^{-1}(\xi),E_{n+1}\>,\quad\xi\in\partial\mathcal{C}_{\omega_0}.
		\end{align*}
		Then for $\xi\in\partial\mathcal{C}_{\omega_0}$, we have
		\begin{align*}
			0=&\<X,E_{n+1}\>\overset{(1)}{=}\<\mathring{\nabla} \hat{s}(\xi) +\hat{s}(\xi)\mathcal{T}^{-1}(\xi),E_{n+1}\>
			\\
			=&\<\mathring{\nabla} \hat{s}(\xi),\nu\>\cdot\<E_{n+1},\nu\>+\<\mathring{\nabla} \hat{s}(\xi),\mu\>\cdot\<E_{n+1},\mu\>-\omega_0\hat{s}(\xi)
			\\
			=&\mathring{g}\(\mathring{\nabla}\hat{s},\mu_{{F}}\)\cdot F(\nu)\<E_{n+1},\mu\>-\omega_0\hat{s},
		\end{align*}
		where we use $E_{n+1}=\<E_{n+1},\nu\>\nu+\<E_{n+1},\mu\>\mu$ (see \cite{Jia-Wang-Xia-Zhang2023}), $\mathring{\nabla}\hat{s}\in T{\mathcal{C}_{\omega_0}}=T\W=T\Sigma$, and the fact that $G(\nu_{{F}})(Y,\mu_{{F}})=\frac{\<Y,\mu\>}{F(\nu)}$ for $\forall Y\in T\Sigma$ (see \cite[eq. (3.29)]{Ding-Gao-Li-arxiv}). The second assertion follows.
			
		(3)	
		By \eqref{equ:Ds}, Lemma \ref{lemma:Gauss-Weingarten}, and  \eqref{equ:Gauss-formula}, for $\xi\in\partial\mathcal{C}_{\omega_0}$, we have at the normal coordinate of $\mathring{g}$
		\begin{align*}
			&\mathring{\nabla}_{e_i}\mathring{\nabla}_{e_j}\hat{s}(\xi)=e_i\left(G(\mathcal{T}^{-1}(\xi))\(D_{e_j} \xi, X\)\right) \\
			= & G(\mathcal{T}^{-1}(\xi))\(D_{e_i}D_{ e_j }\xi, X(\xi)\)+G(\mathcal{T}^{-1}(\xi))\(D_{e_j}( \mathcal{T}^{-1}(\xi)),D_{ e_i} X\)+Q(\mathcal{T}^{-1}(\xi))\(D_{ e_i} \xi, D_{e_j} \xi, X(\xi)\) \\
			= & -\delta_{i j} G(\mathcal{T}^{-1}(\xi))\(\mathcal{T}^{-1}(\xi), X(\xi)\)-\frac{1}{2} Q(\mathcal{T}^{-1}(\xi))\(D_{ e_i} \xi, D_{e_j} \xi, D_{e_k} \xi\) G(\mathcal{T}^{-1}(\xi))\(e_k, X(\xi)\) \\
			& \quad+G(\mathcal{T}^{-1}(\xi))\(\mathcal{T}^{-1}(\xi),\hat{h}_{i j}(X) \mathcal{T}^{-1}(\xi)\)+Q(\mathcal{T}^{-1}(\xi))\(D_{ e_i} \xi, D_{e_j} \xi, D_{e_k} \xi\) G(\mathcal{T}^{-1}(\xi))\(e_k, X(\xi)\) \\
			 =&-\delta_{i j} \hat{s}(\xi)+\hat{h}_{i j}(X(\xi))+\frac{1}{2} Q_{i j k} \mathring{\nabla}_{e_k} \hat{s}(\xi).
		\end{align*}
			
		Hence the anisotropic second fundamental form $\hat{h}_{ij}(X(\xi))$ is given by
		$$\hat{h}_{ij}(X(\xi))=\mathring{\nabla}_{e_i}\mathring{\nabla}_{e_j}\hat{s}(\xi)+\delta_{i j} \hat{s}(\xi)-\frac{1}{2} Q_{i j k} \mathring{\nabla}_{e_k} \hat{s}(\xi).$$
			By the Weingarten formula $D_{ e_i}\xi=D_{ e_i}(\mathcal{T}^{-1}(\xi))=D_{ e_i}\nu_{{F}}(\xi)=\hat{g}^{kl}\hat{h}_{ik}\cdot D_{e_l}X$, we see
			\begin{eqnarray*}
				\mathring{g}_{ij}&=&G(\mathcal{T}^{-1}(\xi))\(D_{ e_i}\xi, D_{e_j}\xi\)
				\\&=&G(\mathcal{T}^{-1}(\xi))\(\hat{g}^{kl}\hat{h}_{ik}\cdot D_{e_l}X, \hat{g}^{pq}\hat{h}_{jp}\cdot D_{e_q}X\)=\hat{h}_{ik}\hat{h}_{jp}\hat{g}^{kp}.
			\end{eqnarray*}
			Hence \begin{eqnarray*}
				\hat{g}_{kp}=\mathring{g}^{ij}\hat{h}_{ik}\hat{h}_{jp}.
			\end{eqnarray*}
			It follows that the principal curvature radii, i.e., the eigenvalues of $\left(\hat{g}_{kp}\hat{h}^{jp}\right)$, are given by the eigenvalues of $\(\mathring{\nabla}_{e_i}\mathring{\nabla}_{e_j}\hat{s}(\xi)+\mathring{g}_{i j} \hat{s}(\xi)-\frac{1}{2} Q_{i j k} \mathring{\nabla}_{e_k} \hat{s}(\xi)\)$ with respect to the metric $\mathring{g}$. Here $\hat{h}^{ij}$ means the inverse of $\hat{h}_{ij}$. The third assertion follows. 		
		\end{proof}		
		\begin{definition}\label{defn-convex capillary fun}
			For $\omega_0\in (-F(E_{n+1}),F(-E_{n+1}))$, a function $f\in  C^{2}(\mathcal{C}_{\omega_0})$  is called an \textbf{anisotropic capillary function} if it satisfies the following   boundary condition
			\begin{eqnarray}
				\label{robin}
				\mathring{\nabla}_{\mu_F} f=\frac{\omega_0}{F(\nu)\<\mu,E_{n+1}\>} f\quad \text{on}\quad \partial \mathcal{C}_{\omega_0}.
			\end{eqnarray}	
			If in addition, $\tau[f](\xi)=\(\tau_{ij}[f](\xi)\):=\(\mathring{\nabla}_{e_i}\mathring{\nabla}_{e_j}f(\xi)+\mathring{g}_{i j} f(\xi)-\frac{1}{2} Q_{i j k} \mathring{\nabla}_{e_k} f(\xi)\)>0$ on $\mathcal{C}_{\omega_0}$, $f$ is called an \textbf{anisotropic capillary convex function} on $\mathcal{C}_{\omega_0}$.
		\end{definition}
		Similar to the isotropic case in \cite[Proposition 2.6]{Xia-arxiv}, the next proposition shows that an anisotropic capillary convex function on $\mathcal{C}_{\omega_0}$  yields an anisotropic  $\omega_0$-capillary convex body in $\overline{\R^{n+1}_+}$, and vice versa. This can be compared to the one-to-one correspondence of convex bodies in ${\R^{n+1}}$ and convex functions on Wulff shape $\W$ (\cite{Xia13}).
		
		\begin{proposition}\label{prop:xia-prpo2.6}
			Let $\hat{s}\in C^{2}(\mathcal{C}_{\omega_0})$. Then $\hat{s}$ is an anisotropic capillary convex function if and only if $\hat{s}$ is the anisotropic capillary support function of an anisotropic capillary convex body $\widehat{\Sigma}\in \mathcal{K}_{\omega_0}$. 
		\end{proposition}
		\begin{proof}
			We have already seen from Lemma \ref{lemma:Xia-Lemma2.4} that
			if $\hat{s}\in C^{2}(\mathcal{C}_{\omega_0})$ is the anisotropic capillary support function of an anisotropic  capillary convex hypersurface $\Sigma\subset \overline{\R^{n+1}_+}$, then  $\hat{s}$ is an anisotropic   capillary convex function.
			It remains to show the converse.
			
			Let $\hat{s}\in C^{2}(\mathcal{C}_{\omega_0})$ be an anisotropic  capillary convex function. Consider the hypersurface $\Sigma$ given by the map
			\begin{eqnarray*}
				X: &&\mathcal{C}_{\omega_0}\to\R^{n+1},\\
				&&X(\xi)=\mathring{\nabla} \hat{s}(\xi)+\hat{s}(\xi) \mathcal{T}^{-1}(\xi).
			\end{eqnarray*}
			
			For  $\xi\in\mathcal{C}_{\omega_0}$, we have $T_{\xi}\mathcal{C}_{\omega_0}=T_{\mathcal{T}^{-1}(\xi)}\W$, then $\mathcal{T}^{-1}(\xi)$ is the anisotropic normal of $\mathcal{C}_{\omega_0}$. It is easy to see $\<\xi, E_{n+1}\>=0$ for $\xi\in\mathcal{C}_{\omega_0}$, then $\<\mathcal{T}^{-1}(\xi),E_{n+1}\>=-\omega_0$ for $\xi\in\mathcal{C}_{\omega_0}$.
			It follows by the decomposition $E_{n+1}=\<E_{n+1},\nu\>\nu+\<E_{n+1},\mu\>\mu$, that
			\begin{align*}
				\<X(\xi), E_{n+1}\>=&\<\mathring{\nabla} \hat{s}(\xi)+\hat{s}(\xi)\mathcal{T}^{-1}(\xi), E_{n+1}\>
				= \<\mathring{\nabla}\hat{s},\mu\>\cdot  \< \mu, E_{n+1}\>-\omega_0\hat{s}(\xi)
				\\
				=&G(\mathcal{T}^{-1}\xi)\(\mathring{\nabla}\hat{s},\mu_{{F}}\)F(\nu)\<\mu,E_{n+1}\>-\omega_0\hat{s}(\xi)
				\\
				\overset{\eqref{robin}}{=}&0, \quad\xi\in\partial\mathcal{C}_{\omega_0},
			\end{align*}
			which implies $X(\p \mathcal{C}_{\omega_0})=\partial\Sigma\subset \partial \overline{\mathbb{R}^{n+1}_{+}}$.
			
			For   any tangent vector field $V=V^ie_i$ on $\mathcal{C}_{\omega_0}$, with normal orthonormal frame $\{e_i\}_{i=1}^n$ on $(\mathcal{C}_{\omega_0},\mathring{g})$,   by Lemma \ref{lemma:Gauss-Weingarten},  we obtain
			\begin{align}\label{derivative}
				D_{V}X(\xi)=&D_{V}(\mathring{\nabla} \hat{s}+\hat{s} \mathcal{T}^{-1}(\xi))
				=V^ie_i\left(e_j(\hat{s})\cdot e_j\right)+V^ie_i\left(\hat{s}\cdot \mathcal{T}^{-1}(\xi)\right)
				\nonumber
				\\=& {V}^i\cdot e_ie_j (\hat{s}) - V^i\cdot e_i(\hat{s})\cdot  \mathcal{T}^{-1}(\xi)
				-\frac{1}{2}V^i\cdot e_j(\hat{s})\cdot  Q_{i j k}\cdot  e_k \nonumber
				\\
				& + V^i\cdot e_i(\hat{s}) \cdot \mathcal{T}^{-1}(\xi)+ \hat{s}(\xi)V\nonumber\\=&\left.\(\mathring{\nabla}^2 \hat{s}+\hat{s}\mathring{g}-\frac{1}{2}Q(\mathcal{T}^{-1}(\xi))(\cdot,\cdot,\mathring{\nabla}\hat{s})\)\right|_{\xi}(V),\quad \forall\xi\in\mathcal{C}_{\omega_0}.	
				 	\end{align}
			Since  $(\tau_{ij}[\hat{s}])>0$, we see $X$ is an immersion.
			Also from \eqref{derivative}, we see that  at $\xi$,
			$$T_{X(\xi)}\Sigma={\rm span}\{D_{e_i}X(\xi)\}_{i=1}^n ={\rm span}\{e_i\}_{i=1}^n .$$
			It follows that
			$\nu_F(X(\xi))=\mathcal{T}^{-1}(\xi)=\xi-\omega_0E_{n+1}^F\in\W$ is the anisotropic normal to $\Sigma$ at $X(\xi)$.
			
			On the other hand, for $\xi\in {\rm int}(\mathcal{C}_{\omega_0})$, we claim that $\<X(\xi), E_{n+1}\>>0$ for each $\xi\in {\rm int}(\mathcal{C}_{\omega_0})$. To see this, 
			for any $\xi_0\in  {\rm int}(\mathcal{C}_{\omega_0})$, assume $\mu_0$ is the unit co-normal  of $\xi_0\in\{x\in\mathcal{C}_{\omega_0}:\<x,E_{n+1}\>=\<\xi_0,E_{n+1}\>\}=\(\mathcal{C}_{\omega_0}\cap\{x_{n+1}=\<\xi_0,E_{n+1}\>\}\)\subset\mathcal{C}_{\omega_0}$ which satisfies $\<\mu_0,E_{n+1}\><0$ (if $E_{n+1}=\nu(\xi)$, take  $\mu_0$ as any unit vector satisfies $\<\mu_0,E_{n+1}\>= 0$). Then we take an arc-length parametrized  curve $\gamma(t):[0,a]\to \mathcal{C}_{\omega_0}$ (see Fig. \ref{fig:1}), which satisfies
			\begin{equation*}
				\renewcommand{\arraystretch}{1.5}
				\left\{
				\begin{array}{ll}
					 \gamma(a)=\xi_a\in\partial\mathcal{C}_{\omega_0},\quad \gamma(0)=\xi_0;
					\\
					\dot{\gamma}(t)=\mu_t;
					\\
					\<\mu_t,E_{n+1}\><0, \,\text{for}\, t\in(0,a].
				\end{array}	
				\right.
			\end{equation*}		
Here $\mu_t$ is the unit co-normal  of $\{x\in\mathcal{C}_{\omega_0}:\<x,E_{n+1}\>=\<\gamma(t),E_{n+1}\>\}\subset\mathcal{C}_{\omega_0}$ at point $\gamma(t)$. In fact $\gamma(t)$ is an  integral curve of $\mu_t$.			
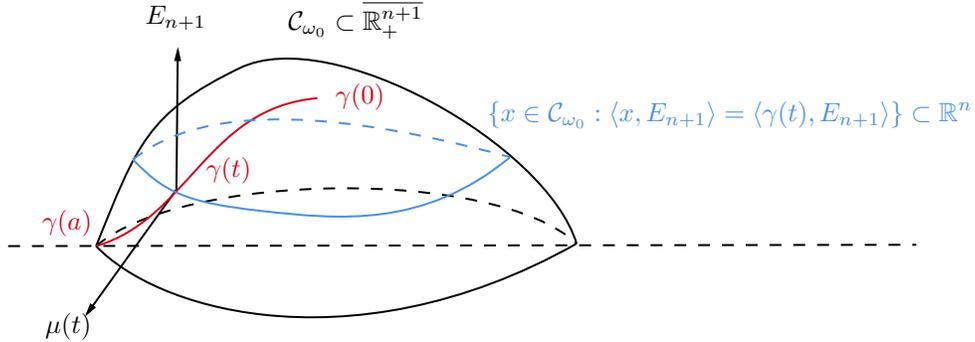
\begin{figure}[h]
				\centering				
				\caption{Curve $\gamma$ and unit co-normal $\mu(t)$}
				\label{fig:1}\hspace{3mm}
			\begin{tikzpicture}[x=0.75pt,y=0.75pt,yscale=-1,xscale=1]
				
				\draw  [dash pattern={on 4.5pt off 4.5pt}]  (72.2,172.26) -- (530.2,171.6) ;
				\draw    (116.56,172.26) .. controls (137.11,121.09) and (142.52,104.86) .. (189.04,82.39) .. controls (235.57,59.93) and (343.75,121.09) .. (358.9,171.01) ;
				\draw    (116.56,172.26) .. controls (132.17,188.52) and (162.3,202.62) .. (200.97,206.89) .. controls (245.39,211.79) and (301.06,203.71) .. (358.9,171.01) ;
				\draw  [dash pattern={on 4.5pt off 4.5pt}]  (116.56,172.26) .. controls (159.83,134.82) and (317.79,132.32) .. (358.9,171.01) ;
				\draw    (157.08,143.72) -- (112.04,206.35) ;
				\draw [shift={(110.87,207.97)}, rotate = 305.73] [fill={rgb, 255:red, 0; green, 0; blue, 0 }  ][line width=0.08]  [draw opacity=0] (7.2,-1.8) -- (0,0) -- (7.2,1.8) -- cycle    ;
				\draw    (157.08,143.72) -- (157.76,73.68) ;
				\draw [shift={(157.78,71.68)}, rotate = 90.56] [fill={rgb, 255:red, 0; green, 0; blue, 0 }  ][line width=0.08]  [draw opacity=0] (7.2,-1.8) -- (0,0) -- (7.2,1.8) -- cycle    ;
				\draw [color={rgb, 255:red, 208; green, 2; blue, 27 }  ,draw opacity=1 ]   (116.56,172.26) .. controls (163.7,158.36) and (172.16,102.55) .. (228.16,97.59) ;
				\draw [color={rgb, 255:red, 74; green, 144; blue, 226 }  ,draw opacity=1 ]   (135.17,128.59) .. controls (150.78,144.86) and (161.53,150.29) .. (200.19,154.55) .. controls (238.86,158.81) and (282.05,163.32) .. (325.37,127.35) ;
				\draw [color={rgb, 255:red, 74; green, 144; blue, 226 }  ,draw opacity=1 ] [dash pattern={on 4.5pt off 4.5pt}]  (135.17,128.59) .. controls (179.55,92.63) and (270.42,112.47) .. (325.37,127.35) ;
				
				\draw (89.41,204.64) node [anchor=north west][inner sep=0.75pt]   [align=left] {$\displaystyle \mu ( t)$};
				\draw (398.42,150.15) node [anchor=north west][inner sep=0.75pt]   [align=left] {};
				\draw (312.44,97.72) node [anchor=north west][inner sep=0.75pt]  [color={rgb, 255:red, 74; green, 144; blue, 226 }  ,opacity=1 ] [align=left] {$\displaystyle \{x\in\mathcal{C}_{\omega_0}:\<x,E_{n+1}\>=\<\gamma(t),E_{n+1}\>\}\subset\mathbb{R}^n$};
				\draw (139.92,49.21) node [anchor=north west][inner sep=0.75pt]   [align=left] {$\displaystyle E_{n+1}$};
				\draw (212.14,47.18) node [anchor=north west][inner sep=0.75pt]   [align=left] {$\displaystyle \mathcal{C}_{\omega_0}\subset \overline{\mathbb{R}_{+}^{
						n+1}}$};
				\draw (170.54,125.65) node [anchor=north west][inner sep=0.75pt]  [color={rgb, 255:red, 208; green, 2; blue, 27 }  ,opacity=1 ] [align=left] {$\gamma(t)$};
				\draw (235.87,88.49) node [anchor=north west][inner sep=0.75pt]  [color={rgb, 255:red, 208; green, 2; blue, 27 }  ,opacity=1 ] [align=left] {$\gamma(0)$};
				\draw (87.92,152.33) node [anchor=north west][inner sep=0.75pt]  [color={rgb, 255:red, 208; green, 2; blue, 27 }  ,opacity=1 ] [align=left] {$\gamma(a)$};

			\end{tikzpicture}
		\end{figure}

			
			Denote by $\nu_t$ the unit outer normal of $\mathcal{C}_{\omega_0}\in\mathbb{R}^{n+1}_+$ at point $\gamma(t)$ for $t\in(0,a]$, and write $\partial\mathcal{C}^t=\{x\in\mathcal{C}_{\omega_0}:\<x,E_{n+1}\>=\<\gamma(t),E_{n+1}\>\}$. Since    vectors $\nu_t,\mu_t,E_{n+1} \in \mathbb{R}^{n+1}$ are all orthogonal to an $(n-1)$-dimensional linear space $T(\partial\mathcal{C}^t)$, then $\nu_t,\mu_t$ and $E_{n+1}$ are on the same plane for $t\in(0,a]$.
			Obviously, $\<\nu_t,\mu_t\>=0$ and $\dot{\gamma}(t)=\mu_t$, and there is the decomposition $E_{n+1}=\<E_{n+1},\dot{\gamma}(t)\>\dot{\gamma}(t)+\<E_{n+1},\nu_t\>\nu_t$ for $t\in [0,a]$.
			 Then we can calculate
			\begin{align*}
			&	\<X(\xi_0), E_{n+1}\>=\<X(\xi_0), E_{n+1}\>-\<X(\xi_a), E_{n+1}\>
				\\=-&\int_0^a\<D_{\dot{\g}(t)}X(\g(t)), E_{n+1}\>dt
				\\=-&\int_0^a \left.\(\mathring{\nabla}^2 \hat{s}+\hat{s}\mathring{g}-\frac{1}{2}Q(\mathcal{T}^{-1}\xi)(\cdot,\cdot,\mathring{\nabla}\hat{s})\)\right|_{\gamma(t)}\(\dot{\gamma}(t), \dot{\gamma}(t)\)\cdot \< E_{n+1}, \dot{\gamma}(t)\>  dt.
			\end{align*}
			Since $(\tau_{ij}[\hat{s}])>0$ and $\< E_{n+1}, \dot{\g}(t)\><0$ for $t\in (0,a ]$, we obtain that $\<X(\xi_0), E_{n+1}\>>0$.
			Hence $X(\mathcal{C}_{\omega_0})\subset \overline{\R^{n+1}_+}$.
			
			Next, we shall prove $X(\mathcal{C}_{\omega_0})=\Sigma\subset \overline{\R^{n+1}_+}$ is a (locally) strictly convex and anisotropic capillary hypersurface.
			 In fact,
			$$\<\nu_F(X(\xi)), E_{n+1}\>=\<\mathcal{T}^{-1}(\xi), E_{n+1}\>=-\omega_0\quad \hbox{for}~\xi\in \p \mathcal{C}_{\omega_0},$$
			which implies that $X(\mathcal{C}_{\omega_0})$ is an anisotropic $\omega_0$-capillary hypersurface.
		By \eqref{derivative}, we have $D_{e_i}X=\sum_{j=1}^{n}\tau_{ij}[\hat{s}]e_j$, then $$D_{e_i}\nu_F=D_{e_i}(\xi-\omega_0E_{n+1}^F)=\(\mathring{\nabla}^2 \hat{s}+\hat{s}\mathring{g}-\frac{1}{2}Q(\mathcal{T}^{-1}(\xi))(\cdot,\cdot,\mathring{\nabla}\hat{s})\)^{-1} \cdot D_{e_j}X.$$
			Hence the eigenvalues of $\left.\(\mathring{\nabla}^2 \hat{s}+\hat{s}\mathring{g}-\frac{1}{2}Q(\mathcal{T}^{-1}(\xi))(\cdot,\cdot,\mathring{\nabla}\hat{s})\)^{-1}\right|_{\xi}$ gives the anisotropic principal curvature of $\Sigma$ at $X(\xi)$. Thus $X: \mathcal{C}_{\omega_0} \to \overline{\R^{n+1}_+}$ is strictly convex.
			
			Finally, we prove that $X: \mathcal{C}_{\omega_0}\to \overline{\R^{n+1}_+}$ is an embedding. For this, we look at the anisotopic capillary Gauss map $\widetilde{\nu_{{F}}}: \Sigma\to \mathcal{C}_{\omega_0}$. From the strictly convexity and anisotropic capillarity of $\Sigma$, we know from 
			Lemma  \ref{Lemma:gaussmap} (2)
			that the restriction of the anisotropic capillary Gauss map $\widetilde{\nu_{{F}}}$ of $\Sigma$ to its boundary  $\p \Sigma$ is one-to-one and furthermore, $\widetilde{\nu_{{F}}}$ is a diffeomorphism on $\Sigma$. 
			On the other hand,
			$\widetilde{\nu_{{F}}}(\Sigma)= \mathcal{C}_{\omega_0}$, and hence $X(\mathcal{C}_{\omega_0})=\Sigma=\widetilde{\nu_F}^{-1}(\mathcal{C}_{\omega_0})$ is an embedding.		
			This completes the proof.
		\end{proof}

\subsection{Mixed volume  for anisotropic capillary convex bodies}

 \begin{proposition}\label{prop:sK=sumSi}
 	Let $K_i \in \mathcal{K}_{\omega_0}, i=1, \cdots, m$. Then the Minkowski sum $K:=\sum_{i=1}^m \lambda_i K_i \in$ $\mathcal{K}_{\omega_0}$, where $\sum_{i=1}^m \lambda_i>0$ and $\lambda_1, \cdots, \lambda_m \geq 0$. Moreover,
\begin{align}
	\label{equ:sK=sumSi}
	\hat{s}_K=\sum_{i=1}^m \lambda_i \hat{s}_{K_i},
\end{align}
where $\hat{s}_K$ and $\hat{s}_{K_i}$ are the anisotropic capillary support functions of $K$ and $K_i$ respectively.
 \end{proposition}
\begin{proof}
	For convex bodies $K_i\in\mathcal{K}_{\omega_0}$, we denote $s_{K_i}(x)$ with $x\in\mathbb{S}^n$ as the classical support function of $K_{i}$. Then the classical support  function of convex bodies $K\in\overline{\mathbb{R}^{n+1}_+}$ is $s_K=\sum_{i=1}^{m}\lambda_is_{K_i}$. Since $\nu_{{F}}=\Psi(\nu)$,  where $\Psi$ is defiend by \eqref{equ:Psi}, the anisotropic support function of $K$ is  $\frac{{s_K}}{F}\circ\Psi^{-1}$ which is defined on $\W$.
	
	Let $\widetilde{s}_{K_i}(\xi)=\frac{s_{K_i}}{F}\circ\Psi^{-1}\circ\mathcal{T}^{-1}(\xi)$, with $\xi\in\(\W+\omega_0E_{n+1}^F\)$. By $\nu_{{F}}=\Psi(\nu)$, $\widetilde{\nu_{{F}}}=\mathcal{T}(\nu_{{F}})$, and the definition of anisotropic capillary support function, we have $\hat{s}_{K_i}=\left.\widetilde{s}_{K_i}\right|_{\mathcal{C}_{\omega_0}}$. We define
	\begin{align*}
		\hat{s}:=\sum_{i=1}^m \lambda_i \hat{s}_{K_i},
		\quad
		\widetilde{s}_K:=\sum_{i=1}^m \lambda_i \widetilde{s}_{K_i}=\frac{s_K}{F}\circ\Psi^{-1}\circ\mathcal{T}^{-1},
	\end{align*}	
	then
$$
\hat{s}:=\left.\widetilde{s}_{K}\right|_{\mathcal{C}_{\omega_0}}.
$$

From Lemma \ref{lemma:Xia-Lemma2.4} (2) and (3), it is easy to see $\hat{s} \in C^2\left(\mathcal{C}_{\omega_0}\right)$ is a capillary convex function, together with Proposition \ref{prop:xia-prpo2.6}, this creates a capillary convex body $L \in \mathcal{K}_{\omega_0}$ with anisotropic capillary support function $\hat{s}$.
From Lemma \ref{lemma:Xia-Lemma2.4} (1), the position vector of $\partial L \cap \overline{\mathbb{R}_{+}^{n+1}}$ is given by
\begin{align}\label{equ:X-xi}
	X(\xi)=\mathring{\nabla} \hat{s}(\xi)+\hat{s}(\xi) \mathcal{T}^{-1}(\xi)=\mathring{\nabla} \widetilde{s}_K(\xi)+\widetilde{s}_K(\xi) \mathcal{T}^{-1}(\xi), \quad \xi \in \mathcal{C}_{\omega_0}.
\end{align}
Let  $\bar{g}(z)=G(z)|_{T_z\W}$  (for  $z\in\W$) be the metric of $(\W,\bar{g},\bar{\nabla})$ (defined in \cite{Xia13}*{Example 2.7}). Since $\mathring{g}:=G(\mathcal{T}^{-1}{(\xi)})|_{T_{{\xi}}\mathcal{C}_{\omega_0}}$ (for ${\xi}\in\mathcal{C}_{\omega_0}\subset\mathcal{T}\W$), we have $\mathring{g}=(\mathcal{T}^{-1}|_{\mathcal{C}_{\omega_0}})^{*}\bar{g}$. Then for any $f\in C^1(\mathcal{C}_{\omega_0})$ we have $\mathring{\nabla}(f\circ\mathcal{T}^{-1})=\bar{\nabla}f\circ\mathcal{T}^{-1}$, so \eqref{equ:X-xi} (the position vector of $L$) can be written as
\begin{align*}
X=\left.\left(\bar{\nabla} (\frac{s_K}{F}\circ\Psi^{-1})(z)+\frac{s_K}{F}\circ\Psi^{-1}(z)\cdot z\right) \right|_{z=\mathcal{T}^{-1}(\xi)}	, \quad \xi \in \mathcal{C}_{\omega_0}.
\end{align*}
By \cite{Xia13}*{Proposition 3.2}, the right hand side is exactly the position vector of $K$ restricting on $\mathcal{C}_{\omega_0}$, hence we conclude that $K=L \in \mathcal{K}_{\omega_0}$. The last assertion follows from $\hat{s}_K=\left.\widetilde{s}_K\right|_{\mathcal{C}_{\omega_0}}$.
\end{proof}

\begin{definition}[\cite{book-convex-body,Xia-arxiv}]
	The mixed discriminant $\mathcal{Q}:\left(\mathbb{R}^{n \times n}\right)^n \rightarrow \mathbb{R}$ is defined by	
	\begin{align}
		\operatorname{det}\left(\lambda_1 A_1+\cdots+\lambda_m A_m\right)=\sum_{i_1, \cdots, i_n=1}^m \lambda_{i_1} \cdots \lambda_{i_n} \mathcal{Q}\left(A_{i_1}, \cdots, A_{i_n}\right),\label{equ:QA}
	\end{align}
	for $m \in \mathbb{N}, \lambda_1, \cdots, \lambda_m \geq 0$ and the real symmetric matrices $A_1, \cdots, A_m \in \mathbb{R}^{n \times n}$. If $\left(A_k\right)_{i j}$ denotes the $(i, j)$-element of the matrix $A_k$, then	
\begin{align}\label{equ:Q1...n}
		\mathcal{Q}\left(A_1, \cdots, A_n\right)=\frac{1}{n!} \sum_{i_1, \cdots, i_n; j_1, \cdots, j_n} \delta_{j_1 \cdots j_n}^{i_1 \cdots i_n}\left(A_1\right)_{i_1 j_1} \cdots\left(A_n\right)_{i_n j_n}.
\end{align}
	
\end{definition}
From \eqref{equ:QA}, we check that for any invertible matrix $B\in \mathbb{R}^{n \times n}$, there holds
\begin{align}\label{equ:Q(AB)=Q(A)det(B)}
	\mathcal{Q}\left(A_1B, \cdots, A_nB\right)=\mathcal{Q}\left(A_1, \cdots, A_n\right)\det(B).
\end{align}
By \eqref{equ:Q1...n}, we have
\begin{align}\label{equ:pQ1...n}
	\frac{\partial}{\partial (A_1)_{ij}}\mathcal{Q}\left(A_1, \cdots, A_n\right)=\frac{1}{n!} \sum_{i, i_2, \cdots, i_n; j, j_2, \cdots, j_n} \delta_{jj_2 \cdots j_n}^{ii_2 \cdots i_n}\left(A_2\right)_{i_2 j_2} \cdots\left(A_n\right)_{i_n j_n}.
\end{align}
Then, we obtain
$$\mathcal{Q}(A_1,\cdots,A_n)=\sum_{i,j=1}^{n}(A_1)_{ij}\cdot \frac{\partial}{\partial (A_1)_{ij}} \mathcal{Q}\left(A_1, \cdots, A_n\right).$$
 Similarly, for fixed $i$, we have
 $$\frac{\partial}{\partial (A_1)_{ii}} \mathcal{Q}\left(A_1, \cdots, A_n\right)=\sum_{k,l=1}^{n}(A_2)_{kl}\cdot \frac{\partial^2}{\partial (A_1)_{ii}\partial (A_2)_{lk}} \mathcal{Q}\left(A_1, \cdots, A_n\right).$$
  If $A_1,\cdots,A_n$ are positive definite, by  an inductive proof (together with a unimodular transformation bringing $A_1$ and $A_2$ to diagonal forms) we know that $\mathcal{Q}(A_1,\cdots,A_n)>0$ and $\left(\frac{\partial}{\partial (A_1)_{ij}} \mathcal{Q}\left(A_1, \cdots, A_n\right)\right)_{n\times n}>0$ (or see \cite{book-convex-body}*{Page 328}).

\begin{definition}\label{def:V(f1,.,fn)}
	For any $f \in C^2\left(\mathcal{C}_{\omega_0}\right)$, set the $n\times n$ matrix $A[f]:=B\cdot \tau[f]\cdot B^T$, where $B$ satisfies $BB^T=\mathring{g}^{-1}$. A multi-variable function $V:\left(C^2\left(\mathcal{C}_{\omega_0}\right)\right)^{n+1} \rightarrow \mathbb{R}$ is defined by	
	\begin{align*}
		V\left(f_0, \cdots, f_{n}\right)=\frac{1}{n+1} \int_{\mathcal{C}_{\omega_0}} f_0 \mathcal{Q}\left(A \left[f_1\right], \cdots, A \left[f_{n}\right]\right) {
		F(\nu)	\mathrm{d} {\mu_{g}}},
	\end{align*}
	for $f_i \in C^2\left(\mathcal{C}_{\omega_0}\right), 0 \leq i \leq n$, where $d\mu_g$ is the induced volume form of the metric $g$ on $\mathcal{C}_{\omega_0}$, which is induced from Euclidean space $(\mathbb{R}^{n+1},\<\cdot,\cdot\>)$. 
\end{definition}
We remark that,  $g$ is denoted as the   metric  of  $\Sigma$ or $\mathcal{C}_{\omega_0}$ induced  from Euclidean space depending on the context.

The importance of the multi-linear function $V$ lies in its symmetry.
The following lemma plays a crucial role in establishing the symmetry of $V(f_1,\cdots,f_{n+1})$ (Lemma \ref{lemma:V12=V21}). This constitutes a fundamental distinction of anisotropic geometry in our work (while in the isotropic case  $\mathcal{Q}^{ij}$ satisfies the divergence-free condition). Specially,  if $f_2=\cdots=f_{n+1}$, \eqref{equ:0=Qjijfi-Q+Q} matches  \cite{Xia13}*{eq. (5.10)}.
\begin{lemma}
	\label{lemma:0=Qjijfi-Q+Q}
	For $f_i\in\C^2(\mathcal{C}_{\omega_0})$, $1\leq i\leq n$, denote by matrices $A_k=\tau[f_k]$,  
	 $\mathcal{Q}:=\mathcal{Q}\left(A_1, \cdots, A_n\right)$, and $\mathcal{Q}^{ij}:=\frac{\partial}{\partial (A_1)_{ij}}\mathcal{Q}\left(A_1, \cdots, A_n\right)$, then
	\begin{align}
		\mathring{\nabla}_j\mathcal{Q}^{ij}\mathring{\nabla}_if_1-\frac{1}{2}\mathcal{Q}^{ij}\mathring{\nabla}_if_1Q_{jkl}\mathring{g}^{kl}+\frac{1}{2}\mathcal{Q}^{ij}\mathring{\nabla}_lf_1Q_{ijk}\mathring{g}^{kl}=0.\label{equ:0=Qjijfi-Q+Q}
	\end{align}
\end{lemma}
\begin{proof}
	Without loss of generality, we compute in an orthogonal frame, i.e., $\mathring{g}_{ij}=\delta_{ij}$. From \cite{Xia13}*{eq.(5.2)}, we see
	\begin{align}\label{equ:pf-lemma:0=Qjijfi-Q+Q-0}
		\mathring{\nabla}_k\tau_{ij}-\mathring{\nabla}_j\tau_{ik}=\sum_{m=1}^{n}\frac{1}{2}Q_{ikm}\tau_{mj}-\frac{1}{2}Q_{ijm}\tau_{mk},\quad \forall k,i,j=1,\cdots, n.
	\end{align}
For any $k,l=2,\cdots,n$, for convenience, we use the symbols $$\mathbb{A}=(A_2)_{i_2j_2}\cdots(A_n)_{i_nj_n},$$ $$\hat{\mathbb{A}}_{k}=(A_2)_{i_2j_2}\cdots(A_{k-1})_{i_{k-1}j_{k-1}} (A_{k+1})_{i_{k+1}j_{k+1}}\cdots (A_n)_{i_nj_n},$$
and for { $k\neq l$}: $$\hat{\mathbb{A}}_{lk}=(A_2)_{i_2j_2}\cdots(A_{k-1})_{i_{k-1}j_{k-1}} (A_{k+1})_{i_{k+1}j_{k+1}}\cdots (A_{l-1})_{i_{l-1}j_{l-1}} (A_{l+1})_{i_{l+1}j_{l+1}}\cdots  (A_n)_{i_nj_n}.$$   	From  \eqref{equ:pQ1...n},  \eqref{equ:pf-lemma:0=Qjijfi-Q+Q-0}, and the antisymmetry of $ \delta_{jj_2\cdots j_n}^{ii_2\cdots i_n} $ with respect to $j$ and $j_k$, we have
	\begin{align}
		&\sum_{i,j=1}^{n}\mathring{\nabla}_j\mathcal{Q}^{ij}\mathring{\nabla}_if_1\nonumber
		\\
		=&
		\sum_{k=2}^{n}\sum_{\substack{i,i_2, \cdots, i_n;\\j, j_2, \cdots, j_n;\\m}}\frac{1}{2n!}\delta_{jj_2\cdots j_n}^{ii_2\cdots i_n}  \mathring{\nabla}_if_1\cdot \left(\mathring{\nabla}_j (A_k)_{i_kj_k}-\mathring{\nabla}_{j_k} (A_k)_{i_kj}\right)\cdot\hat{\mathbb{A}}_k\nonumber
		\\
		\overset{\eqref{equ:pf-lemma:0=Qjijfi-Q+Q-0}}{=}&\sum_{k=2}^{n}\sum_{\substack{i, i_2, \cdots, i_n;\\j, j_2, \cdots, j_n;\\m}}\frac{1}{2n!}\delta_{jj_2\cdots j_n}^{ii_2\cdots i_n}  \mathring{\nabla}_if_1\left(
		\frac{1}{2}Q_{i_kjm}(A_k)_{mj_k}-\frac{1}{2}Q_{i_kj_km}(A_2)_{mj}
		\right)\hat{\mathbb{A}}_k
		\nonumber
		\\
		=&\sum_{k=2}^{n}\sum_{\substack{i, i_2, \cdots, i_n;\\j, j_2, \cdots, j_n;\\m}}\frac{1}{2n!}\delta_{jj_2\cdots j_n}^{ii_2\cdots i_n}  \mathring{\nabla}_if_1\cdot
		Q_{i_kjm}\cdot(A_k)_{mj_k}
		\cdot \hat{\mathbb{A}}_k.\label{equ:pflemma3.17-0}
	\end{align}
 Fixed $k\in\{2,\cdots,n\}$, since the generalized Kronecker symbol $\delta_{jj_2\cdots j_n}^{ii_2\cdots i_n} $ is nonvanishing only when $(i,i_2,\cdots ,i_n)$ is a permutation of $(1,\cdots,n)$,
 it follows that
	
		\begin{align}
		&\sum_{\substack{i, i_2, \cdots, i_n;\\j,j_2, \cdots, j_n;\\m}}
		 \delta_{jj_2\cdots j_n}^{ii_2\cdots i_n}  \mathring{\nabla}_if_1\cdot
		Q_{i_kjm}\cdot(A_k)_{mj_k}
		\cdot \hat{\mathbb{A}}_k\nonumber
		\\
		=&\sum_{\substack{i, i_2, \cdots, i_n;\\j, j_2, \cdots, j_n}}
		\delta_{jj_2\cdots j_n}^{ii_2\cdots i_n}  \mathring{\nabla}_if_1\cdot
		Q_{i_kji}\cdot(A_k)_{ij_k}
		\cdot \hat{\mathbb{A}}_k\nonumber
		\\
		&+\sum_{l=2}^{n}\sum_{\substack{i, i_2, \cdots, i_n;\\j, j_2, \cdots, j_n}}
		\delta_{jj_2\cdots j_n}^{ii_2\cdots i_n}  \mathring{\nabla}_if_1\cdot
		Q_{i_kji_l}\cdot(A_k)_{i_lj_k}
		\cdot \hat{\mathbb{A}}_k\nonumber
		\\
		=:&\sum_{\substack{i, i_2, \cdots, i_n;\\j, j_2, \cdots, j_n}}
		\delta_{jj_2\cdots j_n}^{ii_2\cdots i_n}  \mathring{\nabla}_if_1\cdot
		Q_{i_kji}\cdot(A_k)_{ij_k}
		\cdot \hat{\mathbb{A}}_k
		+\sum_{l=2}^{n}B_{l,k},\label{equ:pflemma3.17-1}
	\end{align}
	where
$$B_{l,k}=:\sum_{\substack{i, i_2, \cdots, i_n;\\j, j_2, \cdots, j_n}}
	\delta_{jj_2\cdots j_n}^{ii_2\cdots i_n}  \mathring{\nabla}_if_1\cdot
	Q_{i_kji_l}\cdot(A_k)_{i_lj_k}
	\cdot \hat{\mathbb{A}}_k,$$
 and obviously
 $$B_{k,k}=\sum_{\substack{i, i_2, \cdots, i_n;\\ j, j_2, \cdots, j_n}}
	\delta_{jj_2\cdots j_n}^{ii_2\cdots i_n}\mathring{\nabla}_if_1\cdot
	Q_{i_kji_k}\cdot \mathbb{A}.$$
 When $k\neq l$,  
  one can check that
\begin{align*}
	B_{l,k}=&\sum_{\substack{i, i_2, \cdots, i_n;\\j, j_2, \cdots, j_n}}
	\delta_{jj_2\cdots j_n}^{ii_2\cdots i_n}  \mathring{\nabla}_if_1\cdot
	Q_{i_kji_l}\cdot(A_k)_{i_lj_k}(A_l)_{i_lj_l}
	\cdot \hat{\mathbb{A}}_{kl}
	\\
	=&\sum_{\substack{i, i_2,\cdots , i_n;\\j, j_2, \cdots, j_n}}
	\delta_{jj_2\cdots j_k\cdots j_l\cdots  j_n}^{ii_2\cdots i_l\cdots i_k\cdots i_n}  \mathring{\nabla}_if_1\cdot
	Q_{i_lji_l}\cdot(A_k)_{i_kj_k}(A_l)_{i_kj_l}
	\cdot \hat{\mathbb{A}}_{kl}
	\\
	=&-B_{k,l},
\end{align*}
where we use the antisymmetry of $ \delta_{jj_2\cdots j_n}^{ii_2\cdots i_n} $ with respect to $i_k$ and $i_{l}$.	Summing both sides of equation \eqref{equ:pflemma3.17-1} over $k\in\{2\cdots,n\}$ yields
	\begin{align}
		&\sum_{k=2}^{n}\sum_{\substack{i, i_2, \cdots, i_n;\\j,j_2, \cdots, j_n;\\m}}
		\delta_{jj_2\cdots j_n}^{ii_2\cdots i_n}  \mathring{\nabla}_if_1\cdot
		Q_{i_kjm}\cdot(A_k)_{mj_k}
		\cdot \hat{\mathbb{A}}_k\nonumber
		\\
		=&\sum_{k=2}^{n}\sum_{\substack{i, i_2, \cdots, i_n;\\j, j_2, \cdots, j_n}}
		\delta_{jj_2\cdots j_n}^{ii_2\cdots i_n}  \mathring{\nabla}_if_1\cdot
		Q_{i_kji}\cdot(A_k)_{ij_k}
		\cdot \hat{\mathbb{A}}_k
		+\sum_{k,l=2}^{n}B_{l,k}\nonumber
		\\
		=&-\sum_{k=2}^{n}\sum_{\substack{i, i_2, \cdots, i_n;\\j, j_2, \cdots, j_n}}\mathring{\nabla}_{i_k}f_1\cdot
		Q_{i_kji}\cdot
		\delta_{jj_2\cdots j_n}^{ii_2\cdots i_n} \cdot\mathbb{A}\nonumber
		\\
		&+\sum_{k=2}^{n}\sum_{\substack{i, i_2, \cdots, i_n;\\ j, j_2, \cdots, j_n}}\mathring{\nabla}_if_1\cdot Q_{i_kji_{k}}\cdot \delta_{jj_2\cdots j_n}^{ii_2\cdots i_n}
		\cdot\mathbb{A}\nonumber
	\\
	=&-\sum_{\substack{i, i_2, \cdots, i_n;\\ j, j_2, \cdots, j_n; \\m}}\mathring{\nabla}_{m}f_1\cdot Q_{mji}\cdot
	\delta_{jj_2\cdots j_n}^{ii_2\cdots i_n}
	\cdot\mathbb{A}
	+\sum_{\substack{i, i_2, \cdots, i_n;\\ j, j_2, \cdots, j_n}}\mathring{\nabla}_{i}f_1\cdot Q_{iji}\cdot
	\delta_{jj_2\cdots j_n}^{ii_2\cdots i_n}
	\cdot\mathbb{A}\nonumber
	\\
		&+\sum_{\substack{i, i_2, \cdots, i_n;\\j, j_2, \cdots, j_n;\\m}}\mathring{\nabla}_if_1\cdot  Q_{mjm}\cdot	\delta_{jj_2\cdots j_n}^{ii_2\cdots i_n}
		\cdot\mathbb{A}-\sum_{\substack{i,i_2, \cdots, i_n;\\j, j_2, \cdots, j_n}}\mathring{\nabla}_if_1\cdot Q_{iji}\cdot	\delta_{jj_2\cdots j_n}^{ii_2\cdots i_n}
		\cdot\mathbb{A}\nonumber
		\\
		=&-\sum_{\substack{i, i_2, \cdots, i_n;\\j, j_2, \cdots, j_n;\\m}}\mathring{\nabla}_{m}f_1\cdot Q_{mji}\cdot
		\delta_{jj_2\cdots j_n}^{ii_2\cdots i_n}
		\cdot\mathbb{A}+\sum_{\substack{i, i_2, \cdots, i_n;\\j, j_2, \cdots, j_n;\\m}}\mathring{\nabla}_if_1\cdot  Q_{mjm}\cdot	\delta_{jj_2\cdots j_n}^{ii_2\cdots i_n}
		\cdot{\mathbb{A}}
		.\label{equ:pf-lem3.17-1}
	\end{align}
  Putting \eqref{equ:pf-lem3.17-1} and \eqref{equ:pQ1...n} in to \eqref{equ:pflemma3.17-0}, we obtain \eqref{equ:0=Qjijfi-Q+Q}.
\end{proof}
Then we can prove following lemma.
\begin{lemma}\label{lemma:V12=V21}
	Let $f_0,\cdots,f_n$ be anisotropic capillary convex functions, then $V(f_0,\cdots,f_n)$ is a symmetric and multi-linear function.
\end{lemma}
\begin{proof}
	The multi-linear property is obvious. By the symmetry of $\mathcal{Q}$, we just need to prove
	\begin{align*}
		V(f_0,f_1,f_2,\cdots,f_n)=V(f_1,f_0,f_2,\cdots,f_n).
	\end{align*}
		By the  $1$-homogeneity of $\mathcal{Q}$, we have
		\begin{align}
			\label{equ:pf-V12=V21-1}
			\mathcal{Q}=\mathcal{Q}^{ij}(A_1)_{ij}.
		\end{align}
		 From \cite{Xia13}*{Lemma 2.8}, for hypersurface $\mathcal{C}_{\omega_0},$ we assume $F(\nu)	\mathrm{d} {\mu_{g}}=\mathrm{d}{\mu}_F=\varphi\mathrm{d}\mu_{\mathring{g}}$, where $\mathrm{d}\mu_{\mathring{g}}$ is the induced volume form of $\mathcal{C}_{\omega_0}$ equipped with $\mathring{g}$, then
		\begin{align}
			\label{equ:pf-V12=V21-2}
			\partial_i\varphi=\varphi\cdot \partial_i\log\varphi=\varphi\sum_{k=1}^{n}A_{ik}^k(\mathcal{T}^{-1}(\xi))=-\frac{1}{2}\varphi\sum_{k=1}^{n}Q_{ikk},
		\end{align}
		where  $\{\partial_i\}_{i=1}^n$ is the  orthogonal frame (i.e., $\mathring{g}_{ij}(\xi)=\delta_{ij}$) which satisfies $e_{n}(\xi)=c(\xi)A_F(\nu(\xi))\mu(\xi)$ when $\xi\in\partial\mathcal{C}_{\omega_0}$. By \cite{Ding-Gao-Li-arxiv}*{Remark 3.1 (1)}, we know $e_{\alpha}\in T_{\xi}(\partial\mathcal{C}_{\omega_0})$ for $\xi\in\partial\mathcal{C}_{\omega_0}$ and $\alpha=1,\cdots,n-1$.
		
  On the other hand, if we let $X(\xi)=\mathring{\nabla}f_1(\xi)+f_1(\xi)\cdot \mathcal{T}^{-1}(\xi)$ for capillary convex function $f_1$ and any $\xi\in\partial\mathcal{C}_{\omega_0}$, then it follows from the proof of Proposition \ref{prop:xia-prpo2.6} that $\<X(\xi),E_{n+1}\>\equiv0$ and
		\begin{align}
			\tau_{\alpha n}[f_1](\xi)\overset{\eqref{derivative}}{=}&G(\mathcal{T}^{-1}(\xi))\(D_{e_{\alpha}}X,e_n\)=\frac{c(\xi)\<D_{e_{\alpha}}X,\mu\>}{F(\nu)}=\frac{c(\xi)\<D_{e_{\alpha}}X,E_{n+1}\>}{F(\nu)\<E_{n+1},\mu\>}\nonumber
			\\
			=&\frac{c(\xi){e_{\alpha}}(\<X,E_{n+1}\>)}{F(\nu)\<E_{n+1},\mu\>}=0, \label{equ:pf-V12=V21-3}	
		\end{align}
		where one uses the fact that $D_{e_{\alpha}}X\in T \mathcal{C}_{\omega_0}$ (by \eqref{derivative}),  $G(\nu_{{F}}(\xi))\(Y,A_F(\nu)\mu\)=\frac{\<Y,\mu\>}{F(\nu)}$ for any $ Y\in T \mathcal{C}_{\omega_0}$ (see \cite{Ding-Gao-Li-arxiv}*{eq. (3.29)}), and $E_{n+1}=\<E_{n+1},\nu\>\nu+\<E_{n+1},\mu\>\mu$.
		
 By Definition \ref{def:V(f1,.,fn)}, \eqref{equ:pf-V12=V21-1}$\thicksim$\eqref{equ:pf-V12=V21-3}, and divergence theorem, we have
	\begin{align*}
		&V\left(f_0, \cdots, f_{n}\right)=\frac{1}{n+1} \int_{\mathcal{C}_{\omega_0}} f_0 \mathcal{Q}\left(\tau \left[f_1\right], \cdots, \tau \left[f_{n}\right]\right) F(\nu)	\mathrm{d} {\mu_{g}}
		\\
		=&\frac{1}{n+1} \int_{\mathcal{C}_{\omega_0}}f_0\mathcal{Q}^{ij}\(\mathring{\nabla}_{i}\mathring{\nabla}_{j}f_1+\delta_{i j} f_1-\frac{1}{2}\sum_{k=1}^{n} Q_{i j k} \mathring{\nabla}_{k} f_1\)\varphi\mathrm{d}\mu_{\mathring{g}}
		\\
		=&\frac{1}{n+1} \int_{\mathcal{C}_{\omega_0}}
		-\varphi \mathcal{Q}^{ij}\mathring{\nabla}_{i}f_0 \mathring{\nabla}_{j}f_1- \varphi f_0\mathring{\nabla}_{j}\mathcal{Q}^{ij} \mathring{\nabla}_{i}f_1
		-f_0 \mathcal{Q}^{ij}\mathring{\nabla}_{j}f_1\partial_i\varphi
		\\
		&+\mathcal{Q}^{ij}\delta_{i j} f_1f_0\varphi-\frac{1}{2}\sum_{k=1}^{n}\mathcal{Q}^{ij} Q_{i j k} f_0\mathring{\nabla}_{k} f_1\varphi\ \mathrm{d}\mu_{\mathring{g}}
		+\frac{1}{n+1}\int_{\partial\mathcal{C}_{\omega_0}} f_0\varphi\mathcal{Q}^{in}\mathring{\nabla}_if_1 d \mathring{s}
		\\
		\overset{\eqref{robin}}{=}&\frac{1}{n+1} \int_{\mathcal{C}_{\omega_0}}
		-\varphi \mathcal{Q}^{ij}\mathring{\nabla}_{i}f_0 \mathring{\nabla}_{j}f_1
		+\mathcal{Q}^{ij}\delta_{i j} f_1f_0\varphi
		\\
		&-f_0\varphi\( \mathring{\nabla}_{j}\mathcal{Q}^{ij} \mathring{\nabla}_{i}f_1
		-\frac{1}{2}\sum_{k=1}^{n}Q_{ikk} \mathcal{Q}^{ij}\mathring{\nabla}_{j}f_1+\frac{1}{2}\sum_{k=1}^{n}\mathcal{Q}^{ij} Q_{i j k} \mathring{\nabla}_{k} f_1\)\ \mathrm{d}\mu_{\mathring{g}}
		\\
		&+\frac{1}{n+1}\int_{\partial\mathcal{C}_{\omega_0}} f_0\varphi\mathcal{Q}^{nn}\frac{c(\xi)\omega_0}{F(\nu(\xi))\<\mu(\xi),E_{n+1}\>}f_1 d \mathring{s}
		\\
			\overset{\eqref{equ:0=Qjijfi-Q+Q}}{=}&
		\frac{1}{n+1} \int_{\mathcal{C}_{\omega_0}}
		- \mathcal{Q}^{ij}\mathring{\nabla}_{i}f_0 \mathring{\nabla}_{j}f_1
		+\mathcal{Q}^{ij}\delta_{i j} f_1f_0 F(\nu)	\mathrm{d} {\mu_{g}}
		\\
		&+\frac{1}{n+1}\int_{\partial\mathcal{C}_{\omega_0}} f_0f_1 \varphi\mathcal{Q}^{nn}\frac{c(\xi)\omega_0}{F(\nu(\xi))\<\mu(\xi),E_{n+1}\>}d \mathring{s},
	\end{align*}
	which   is symmetric in $f_0$ and $f_1$, where $d \mathring{s}$ is the  area element induced on the boundary $\partial \mathcal{C}_{\omega_0}$
	from  $(\mathcal{C}_{\omega_0},\mathring{g})$. We finish the proof of the lemma.
\end{proof}
\begin{remark}
	By adapting the technique of  \cite{Xia-arxiv}*{Corollary 2.10}, since  $\hat{s}_o$ (defined by \eqref{equ:s0}) is anisotropic capillary support function of $\mathcal{C}_{\omega_0}$  and $\tau[\hat{s}_o]=I_n$ (see \eqref{equ:A[l]}),
 the anisotropic capillary Minkowski formula (see \cite{Jia-Wang-Xia-Zhang2023}*{Theorem 1.3}) can also be derived from Proposition \ref{lemma:V12=V21}.
\end{remark}

 In terms of the construction of mixed volumes for standard capillary convex bodies in  \cite{Xia-arxiv}, we similarly consider the combination of anisotropic capillary convex bodies $K:=\sum_{i=1}^m \lambda_i K_i$  for $K_1, \cdots, K_m \in \mathcal{K}_{\omega_0}$.  The mixed volume $V\left(K_{i_0}, \cdots, K_{i_{n}}\right)$ defined by
\begin{align}
	|K|=\sum_{i_0, \cdots, i_{n}=1}^m \lambda_{i_0} \cdots \lambda_{i_{n}} V\left(K_{i_0}, \cdots, K_{i_{n}}\right),\label{equ:|K|}
\end{align}
for anisotropic capillary convex bodies in $\mathcal{K}_{\omega_0}$ can be written as an integral over $\mathcal{C}_{\omega_0}$ by using the anisotropic  capillary support function.

\begin{proposition}\label{prop7.3}
	Let $K_1, \cdots, K_m \in \mathcal{K}_{\omega_0}$ and $\hat{s}_i$ be the anisotropic  capillary support function of $K_i$. Then we have
	\begin{align}\label{equ:VKKKLLL}
		V\left(K_{i_0}, \cdots, K_{i_{n}}\right)&=V\left(\hat{s}_{i_0}, \cdots, \hat{s}_{i_{n}}\right)
		\\
		&=\frac{1}{n+1} \int_{\mathcal{C}_{\omega_0}} \hat{s}_{i_0} \mathcal{Q}\left(A\left[\hat{s}_{i_2}\right], \cdots, A\left[\hat{s}_{i_{n}}\right]\right) F(\nu)	\mathrm{d} {\mu_{g}}
	   .\nonumber
	\end{align}
\end{proposition}
\begin{proof}
Denote $K:=\sum_{i=1}^m \lambda_i K_i$ and $\Sigma=\partial K\cap \mathbb{R}_+^{n+1}$.	Since $\widetilde{\nu_F}:\Sigma\to \mathcal{C}_{\omega_0}$ is a diffeomorphism and $\hat{s}_K(\xi)=\hat{u}_K(\widetilde{\nu_{{F}}}^{-1}(\xi))$ for $\xi\in\mathcal{C}_{\omega_0}$, we have 
	\begin{align*}
		(n+1)|K|=\int_{\Sigma}\<X,\nu\>\mathrm{~d}\mu_g=\int_{\Sigma}\frac{\<X,\nu\>}{F(\nu)}\mathrm{~d}\mu_F=\int_{\mathcal{C}_{\omega_0}}
		\hat{s}_K\det(A[\hat{s}_K])F(\nu)	\mathrm{d} {\mu_{g}}.
	\end{align*}
	By  \eqref{equ:sK=sumSi}, $\hat{s}_K=\sum_{i=1}^{m}\lambda_i\hat{s}_i$, then we obtain
		\begin{align*}
		(n+1)|K|=&\int_{\mathcal{C}_{\omega_0}}
		\hat{s}_K\det(A[\hat{s}_K])F(\nu)	\mathrm{d} {\mu_{g}}
		\\
		=&\int_{\mathcal{C}_{\omega_0}}
		\(\sum_{i=1}^{m}\lambda_i\hat{s}_i\)\det(\lambda_iA[\hat{s}_1]+\cdots\lambda_mA[\hat{s}_m])F(\nu)	\mathrm{d} {\mu_{g}}
		\\
		\overset{\eqref{equ:QA}}{=}&\int_{\mathcal{C}_{\omega_0}}
		\(\sum_{i=1}^{m}\lambda_i\hat{s}_i\)
		\(\sum_{i_1, \cdots, i_n=1}^m \lambda_{i_1} \cdots \lambda_{i_n} \mathcal{Q}\left(A[{\hat{s}_{i_1}}], \cdots, A[\hat{s}_{i_n}]\right)\)
		F(\nu)	\mathrm{d} {\mu_{g}}
		\\
		=&\sum_{i_0, \cdots, i_n=1}^m \lambda_{i_0} \cdots \lambda_{i_n} \int_{\mathcal{C}_{\omega_0}} \hat{s}_{i_0} \mathcal{Q}
		\left(A[{\hat{s}_{i_1}}], \cdots, A[\hat{s}_{i_n}]\right)
	F(\nu)	\mathrm{d} {\mu_{g}}.
	\end{align*}
	Combining with \eqref{equ:|K|}, we have \eqref{equ:VKKKLLL}.
\end{proof}

Specially, from \cite{Ding-Gao-Li-arxiv}*{Lemma 7.6 (i)}, we can prove the following lemma.
\begin{lemma}\label{lemma:V-k=V(KKK,LLL)}
	Let $\widehat{\Sigma}\in\mathcal{K}_{\omega_0}$ be an anisotropic capillary convex body with anisotropic capillary support function $\hat{s}$ in $\mathcal{C}_{\omega_0}$. Then for $-1\leq k\leq n$, we have
	\begin{align*}
		\mathcal{V}_{k+1,\omega_0}(\Sigma)	=V\(\underbrace{\widehat{\Sigma}, \cdots, \widehat{\Sigma}}_{(n-k) \text { copies }}, \underbrace{\widehat{\mathcal{C}}_{\omega_0}, \cdots, \widehat{\mathcal{C}}_{\omega_0}}_{(k+1) \text { copies }}\).
	\end{align*}
\end{lemma}
\begin{proof}
	For $k=-1$, it is easy to check that
	\begin{align*}
		V\left(\widehat{\Sigma},\cdots,\widehat{\Sigma}\right)=&\frac{1}{n+1} \int_{\mathcal{C}_{\omega_0}} \hat{s} \mathcal{Q}\left(A\left[\hat{s}\right], \cdots, A\left[\hat{s}\right]\right) F(\nu)	\mathrm{d} {\mu_{g}}
		\\
		=&\frac{1}{n+1} \int_{\mathcal{C}_{\omega_0}} \hat{s} \det(A[\hat{s}]) F(\nu)	\mathrm{d} {\mu_{g}}
		\\
		=&\frac{1}{n+1} \int_{\Sigma} \hat{u}  d {\mu}_F
		\\
		=&|\widehat{\Sigma}|=\mathcal{V}_{0,\omega_0}(\Sigma),
	\end{align*}
 where we use	the fact that the support function of hypersurface  ${\partial\Sigma}\subset\partial\overline{\mathbb{R}^{n+1}_+}$ in $\mathbb{R}^{n+1}$ is everywhere zero.

 Now we conside the case  $k=0,\cdots,n$.
	Since the anisotropic support function of $\W$ is constant $1$, the anisotropic support function of $\widetilde{\W}=\W+\omega_0E_{n+1}^F$ is $$\hat{u}_0(X)=1+\omega_0G(\nu_{{F}}(X))\(E_{n+1}^F,\nu_{{F}}(X)\),$$ for $X\in\widetilde{\W}$. Then the anisotropic capillary support function of $\mathcal{C}_{\omega_0}=\widetilde{\W}\cap\overline{\mathbb{R}^{n+1}_+}$ is
$$\hat{s}_o(\xi)=\hat{u}_0\left(\widetilde{\nu_F}^{-1}(\xi)\right)=1+\omega_0G(\mathcal{T}^{-1}(\xi))\(E_{n+1}^F,\mathcal{T}^{-1}(\xi)\),\quad \xi\in\mathcal{C}_{\omega_0}.$$
Choose the normal coordinate $\{e_i=\partial_i\xi\}_{i=1}^n\in T_{\xi}\mathcal{C}_{\omega_0}$, by Lemma \ref{lemma:Gauss-Weingarten} we have
\begin{align}
	\mathring{\nabla}_{e_i}\left(
	G(\mathcal{T}^{-1}(\xi))\(E_{n+1}^F,\mathcal{T}^{-1}(\xi)\)
	\right)=&G(\mathcal{T}^{-1}(\xi))\(E_{n+1}^F,\partial_i(\mathcal{T}^{-1}(\xi))\)\nonumber
	\\
	=&G(\mathcal{T}^{-1}(\xi))\(E_{n+1}^F,e_i\),\label{equ:pf-lemma3.20-1}
\end{align}
and
\begin{align}
	&\mathring{\nabla}_{e_j}\mathring{\nabla}_{e_i}\left(
	G(\mathcal{T}^{-1}(\xi))\(E_{n+1}^F,\mathcal{T}^{-1}(\xi)\)
	\right)=\mathring{\nabla}_{e_j}\left(G(\mathcal{T}^{-1}(\xi))\(E_{n+1}^F,e_i\)\right)\nonumber
	\\
	=&G(\mathcal{T}^{-1}(\xi))\(E_{n+1}^F,\partial_j\partial_i\xi\)+Q(\mathcal{T}^{-1}(\xi))\(E_{n+1}^F,e_i,\partial_j(\mathcal{T}^{-1}(\xi))\)\nonumber
	\\
	=&G(\mathcal{T}^{-1}(\xi))\(E_{n+1}^F,-\delta_{i j} \mathcal{T}^{-1}(\xi)-\frac{1}{2}  Q_{i j k} e_k\)
	+Q(\mathcal{T}^{-1}(\xi))\(E_{n+1}^F,e_i,e_j\)	\nonumber
	\\
	=&-\delta_{i j}G(\mathcal{T}^{-1}(\xi))\(E_{n+1}^F, \mathcal{T}^{-1}(\xi)\)
	\label{equ:pf-lemma3.20-2}+\frac{1}{2}  Q_{i j k}G(\mathcal{T}^{-1}(\xi))\(E_{n+1}^F, e_k\).
\end{align}
Combination of \eqref{equ:pf-lemma3.20-1} and \eqref{equ:pf-lemma3.20-2} gives
\begin{align}\label{equ:A[X,v]}
	\tau_{ij}\left[G(\mathcal{T}^{-1}(\xi))\(E_{n+1}^F,\mathcal{T}^{-1}(\xi)\)\right]=0,
\end{align}
and then
\begin{align}\label{equ:A[l]}
	\tau[\hat{s}_o]=\tau[1]=[\delta_{i j}]=:I_n.
\end{align}
	
	From \cite{Ding-Gao-Li-arxiv}*{Lemma 7.6 (i)}, we can calculate that
	\begin{align*}
		\mathcal{V}_{k+1,\omega_0}(\Sigma)=&\frac{1}{n+1}\int_{\Sigma}H_k^F\(1+\omega_0G(\nu_{{F}})(\nu_{{F}},E_{n+1}^F)\) {\rm d}\mu_F
		\\
		=&\frac{1}{n+1}\int_{\mathcal{C}_{\omega_0}}\hat{s}_o\mathcal{Q}(
		 (\underbrace{\tau[\hat{s}],\cdots,\tau[\hat{s}]}_{(n-k) \text{ copies }},\underbrace{I_n,\cdots,I_n}_{k \text{ copies }})
		) F(\nu)	\mathrm{d} {\mu_{g}}
		\\
		=&V\(\underbrace{\hat{s}, \cdots, \hat{s}}_{(n-k) \text { copies }}, \underbrace{\hat{s}_o, \cdots, \hat{s}_o}_{(k+1) \text { copies }}\).
	\end{align*}
	Then we complete the proof.
\end{proof}
\begin{remark}\label{rk:3}
	In the proof process of \eqref{equ:A[X,v]}, it can be observed that $E_{n+1}^F$ in \eqref{equ:A[X,v]} may be substituted with an arbitrary constant vector in $\mathbb{R}^{n+1}$.
\end{remark}
From \eqref{equ:|K|} and Lemma \ref{lemma:V-k=V(KKK,LLL)}, we have a direct consequence (Steiner-type formula) as follows, which extends the result of \cite{Xia-arxiv}*{Proposition 2.16} to anisotropic capillary convex bodies.
\begin{lemma}
	Let $\widehat{\Sigma}\in\mathcal{K}_{\omega_0}$ be associated with a convex anisotropic $\omega_0$-capillary hypersurface $\Sigma$. Then for any $t>0$, there holds
	\begin{align*}
		|\widehat{\Sigma}+t\widehat{\mathcal{C}}_{\omega_0}|=\sum_{k=0}^{n+1}\binom{n+1}{k}t^k\mathcal{V}_{k,\omega_0}(\Sigma).
	\end{align*}
\end{lemma}

\section{The Alexandrov-Fenchel inequalities for mixed volume}\label{sec 4}
Based on the preparations in Section \ref{sec 3} (including definitions and properties of anisotropic capillary convex bodies in the anisotropic setting), we can prove the main Theorem \ref{thm:A-F-convex} following the process in \cite{Xia-arxiv}, or adapting the idea of \cite{Shenfeld19,Shenfeld22}.

We first derive the following lemma, which is inspired by \cite{Xia13}*{Lemma 5.4}.
\begin{lemma}
	\label{lemma:ker-A}
	Let $f(\xi)\in C^2(\mathcal{C}_{\omega_0})$ be a function such that $\tau[f]=0$. Then
	\begin{align}\label{equ:f=1-n+1}
		f=\sum_{\alpha=1}^{n+1}a_{\alpha}G(\mathcal{T}^{-1}(\xi))(\mathcal{T}^{-1}(\xi),E_{\alpha}),
	\end{align}
	for some constants $a_1,\cdots,a_{n+1}$.
	Furthermore, if $f$ also satisfies the boundary condition \eqref{robin}, then $a_{n+1}=0$, that is
	\begin{align*}
		f=\sum_{i=1}^{n}a_{i}G(\mathcal{T}^{-1}(\xi))(\mathcal{T}^{-1}(\xi),E_{i}),
	\end{align*}
	for some constants $a_1,\cdots,a_{n}$.
\end{lemma}
\begin{proof}
	Let $e_{n+1}(\xi)=\mathcal{T}^{-1}(\xi)$ be the anisotropic normal vector of $\xi\in\mathcal{C}_{\omega_0}$ and $\{e_i(\xi)\}_{i=1}^n$ a local orthonormal frame field with respect to $\mathring{g}$ on $\mathcal{C}_{\omega_0}$ such that $\{e_{\alpha}\}_{\alpha=1}^{n+1}$ is positive-oriented orthonormal frame field with respect  to $G(\mathcal{T}^{-1}(\xi))$ in $T_{\xi}\mathbb{R}^{n+1}$ ($\xi\in\mathcal{C}_{\omega_0}$). By Lemma \ref{lemma:Gauss-Weingarten}, we have (in normal coordinates at a point, i.e. $\mathring{g}_{ij}=\delta_{ij}$ and $\mathring{\nabla}_{e_i}e_j=0$)
	\begin{align}\label{equ:ei-en+1}
	e_i(e_{n+1})=e_i, \quad e_{i}(e_j)=-\frac{1}{2}\sum_{k=1}^nQ_{ijk}e_k-\delta_{ij}e_{n+1},
	\end{align}
	for $ i,j\in\{1,\cdots,n\}$. Denote by $f_i=e_i(f)$, and define a vector valued function \begin{align}\label{equ:Z=}
		Z(\xi)=\sum_{i=1}^{n}f_i(\xi)e_i(\xi)+f(\xi)e_{n+1}(\xi),\quad \xi\in\mathcal{C}_{\omega_0}.
	\end{align} Since $G(\mathcal{T}^{-1}(\xi))\(e_i,e_{n+1}\)=0$ and $G(\mathcal{T}^{-1}(\xi))\(e_{n+1},e_{n+1}\)=1$, we get
	\begin{align}\label{equ:f=}
	f=G(\mathcal{T}^{-1}(\xi))\(Z,e_{n+1}\).
	\end{align}
	From \eqref{equ:ei-en+1},  \eqref{equ:Z=} and $\tau[f]=0$, one derives
	\begin{align*}
		e_i(Z)=\sum_{j=1}^n\tau_{ij}[f]\cdot e_j=0.
	\end{align*}
	Then $Z(\xi)\in\mathbb{R}^{n+1}$ is a constant vector for $\xi\in\mathcal{C}_{\omega_0}$, and can be written as $Z=\sum_{\alpha=1}^{n+1}a_{\alpha}E_{\alpha}$ for some constants $a_{\alpha}$. Putting it into \eqref{equ:f=} yields \eqref{equ:f=1-n+1}.
	
	If in addition $f$ satisfies \eqref{robin}, we have
	\begin{align}
		\frac{\omega_0}{F(\nu)\<\mu,E_{n+1}\>} f
		=&
		\mathring{\nabla}_{\mu_F} f
		=\mathring{g}(\mathring{\nabla}f,\mu_{{F}})=G(\mathcal{T}^{-1}(\xi))\(\mathring{\nabla}f,\mu_{{F}}\)\nonumber
		\\
		=&\frac{\<\mathring{\nabla}{f},\mu\>}{F(\nu)}
		\quad \text{on}\quad \partial \mathcal{C}_{\omega_0},\nonumber
	\end{align}
	where the  last equality follows from \cite{Ding-Gao-Li-arxiv}*{eq. (3.29)}.
	Combining with $Z=\sum_{\alpha=1}^{n+1}a_{\alpha}E_{\alpha}$, \eqref{equ:Z=}, and $E_{n+1}=\<E_{n+1},\nu\>\nu+\<E_{n+1},\mu\>\mu$, 
	 we have
	\begin{align*}
		a_{n+1}=&\<Z,E_{n+1}\>=\<fe_{n+1},E_{n+1}\>+\<\mathring{\nabla}f,\mu\>\<\mu,E_{n+1}\>
	\\
	=&-\omega_0f+\omega_0f=0.
	\end{align*}
	Here we use the fact that $\widehat{\mathcal{C}}_{\omega_0}\in\mathcal{K}_{\omega_0}$, which means $-\omega_0=\<\nu_{{F}}(\xi),E_{n+1}\>=\<\mathcal{T}^{-1}(\xi),E_{n+1}\>$.
	This completes the proof.	
\end{proof}

The following lemma will be crucially used in the subsequent analysis.

\begin{lemma}[ \cite{Shenfeld19}*{Lemma 1.4}]\label{lemma:3.1}
	 Let $\mathcal{A}$ be a symmetric matrix. Then the following conditions are equivalent:
\begin{itemize}
	\item [(i)] $\langle x, \mathcal{A} y\rangle^2 \geq\langle x, \mathcal{A} x\rangle\langle y, \mathcal{A} y\rangle$ for all $x, y$ such that $\langle y, \mathcal{A} y\rangle \geq 0$;
\item[(ii)] The positive eigenspace of $\mathcal{A}$ has dimension at most one.
\end{itemize}
The conclusion remains valid if $\mathcal{A}$ is a self-adjoint operator on a Hilbert space with a discrete spectrum, provided the vectors $x, y$ are chosen in the domain of $\mathcal{A}$.
\end{lemma}

We now establish the following result:
 \begin{theorem}\label{thm:V2>VV}
 	Let $f_1, \cdots, f_n \in C^2\left(\mathcal{C}_{\omega_0}\right)$ be anisotropic capillary convex functions on $\mathcal{C}_{\omega_0}$ and at least one of $f_k\,(2 \leq k \leq n)$ is positive. Then for any anisotropic capillary function $f \in C^2\left(\mathcal{C}_{\omega_0}\right)$, there holds
\begin{align}\label{equ:V2>VV}
	V^2\left(f, f_1, f_2, \cdots, f_n\right) \geq V\left(f, f, f_2, \cdots, f_n\right) V\left(f_1, f_1, f_2, \cdots, f_n\right) .
\end{align}
Equality holds if and only if $f(\xi)=a f_1(\xi)+\sum_{i=1}^n a_iG(\mathcal{T}^{-1}(\xi))\(\mathcal{T}^{-1}(\xi), E_i\)$ for some constants $a, a_i \in \mathbb{R}$. 
 \end{theorem}
\begin{proof} Without loss of generality, we assume that $f_2$ is positive on $\mathcal{C}_{\omega_0}$. Define $\omega$ to be the following area measure on $\mathcal{C}_{\omega_0}$
$$
d \omega:=\frac{1}{n+1} \frac{Q\left(A\left[f_2\right], A\left[f_2\right], A\left[f_3\right],\cdots, A\left[f_n\right]\right)}{f_2} F(\nu)	\mathrm{d} {\mu_{g}} .
$$
We assume that  $L^2\left(\mathcal{C}_{\omega_0}, \omega\right)$ is the Hilbert space associated with the inner product $\langle,\rangle_{L^2(\omega)}$ given by
$$
\langle f, g\rangle_{L^2(\omega)}:=\int_{\mathcal{C}_{\omega_0}} f \cdot g d \omega, \quad \forall f, g \in L^2\left(\mathcal{C}_{\omega_0}, \omega\right) .
$$
Similarly we can define $W^{2,2}\left(\mathcal{C}_{\omega_0}, \omega\right)$ in terms of the inner product  $\langle,\rangle_{L^2(\omega)}$. Define the following operator similar to \cite{Xia-arxiv,Shenfeld19}:
$$
\begin{aligned}
\mathcal{A}:  \operatorname{dom}	&(\mathcal{A}) \subset L^2\left(\mathcal{C}_{\omega_0}, \omega\right) \rightarrow L^2\left(\mathcal{C}_{\omega_0}, \omega\right), \\
	 \mathcal{A}(f)=:&\frac{f_2 \mathcal{Q}\left(A[f], A\left[f_2\right], \cdots, A\left[f_n\right]\right)}{\mathcal{Q}\left(A\left[f_2\right], A\left[f_2\right], \cdots, A\left[f_n\right]\right)}
	\\
	=&\frac{f_2\mathcal{Q}^{ij}\tau_{ij}[f]}{\mathcal{Q}(A[f_2],A[f_2],\cdots,A[f_n])}
	\\
	=:&a^{ij}(\xi)\mathring{\nabla}_{i}\mathring{\nabla}_{j}f+b^i(\xi)\mathring{\nabla}_{i}f+c(\xi)f,\quad \xi\in \mathcal{C}_{\omega_0},
\end{aligned}
$$
where $a^{ij}=\frac{f_2\mathcal{Q}^{ij}}{\mathcal{Q}(A[f_2],A[f_2],\cdots,A[f_n])}$ is independent of $f$, and $\operatorname{dom}(\mathcal{A})$ is given by
$$
\operatorname{dom}(\mathcal{A}):=\left\{f \in W^{2,2}\left(\mathcal{C}_{\omega_0}, \omega\right): f\, \text{satisfies boundary condition \eqref{robin}}\right\} .
$$
Since $f_2>0$ and $f_2, \cdots, f_n$ are anisotropic capillary convex functions on $\mathcal{C}_{\omega_0}$, all the above objects are well-defined.

By the very definition of $\mathcal{A}$, \eqref{equ:V2>VV} is equivalent to
\begin{align}\label{equ:AAf2>fAffAf}
	\left(\left\langle f, \mathcal{A} f_1\right\rangle_{L^2(\omega)}\right)^2 \geq\langle f, \mathcal{A} f\rangle_{L^2(\omega)}\left\langle f_1, \mathcal{A} f_1\right\rangle_{L^2(\omega)},
\end{align}
for functions $f,\left\{f_i\right\}_{i=1}^n$ that satisfy the assumptions. Lemma \ref{lemma:3.1} reduces the proof to show that the positive eigenspace of $\mathcal{A}$ is at most one-dimensional.

Since $f_1, \cdots, f_n$ are anisotropic capillary convex functions and $f_2>0$, we know $[\mathcal{Q}^{ij}]>0$, and $[a^{ij}]>0$, then $\mathcal{A}$ is a uniformly elliptic operator with a Robin boundary condition. Lemma  \ref{lemma:V12=V21} implies
$$
\langle f, \mathcal{A} g\rangle_{L^2(\omega)}=\langle g, \mathcal{A} f\rangle_{L^2(\omega)}, \quad \forall f, g \in \operatorname{dom}(\mathcal{A}),
$$
which means that $\mathcal{A}$ is a self-adjoint operator on $\operatorname{dom}(\mathcal{A})$. Consider the following eigenvalue problem
\begin{eqnarray}\label{equ:eigenvalue-problem}
	\left\{
\begin{aligned}
	\mathcal{A} f & =\lambda f, & & \text { in } \mathcal{C}_{\omega_0}, \\
	\mathring{\nabla}_{\mu_F} f&=\frac{\omega_0}{F(\nu)\<\mu,E_{n+1}\>} f.\quad  & & \text { on } \partial \mathcal{C}_{\omega_0} .
\end{aligned}\right.
\end{eqnarray}
According to the classical spectral theory for compact, self-adjoint operators, the eigenvalue problem \eqref{equ:eigenvalue-problem} admits a sequence of real eigenvalues $\lambda_1 \geq \lambda_2 \geq \cdots$, where the first eigenvalue $\lambda_1$ is simple and its first eigenfunction has the same sign everywhere (see \cite{Chang-Wang-Wu}*{Theorem 5.10}, or see Krein-Rutman Theorem in \cite{Deimling}*{Theorem 19.3}, \cite{Du2006}*{Theorem 1.2}).
It is trivial to check that $f_2$ satisfies \eqref{equ:eigenvalue-problem} with $\lambda=1$. Since $f_2>0, f_2$ is the first eigenfunction with the first eigenvalue $\lambda_1=1$.

On the other hand, since $A[f_2],\cdots,A[f_n]$ are positive definite, by using Alexandrov's mixed discriminant inequality (see e.g. \cite{book-convex-body}*{Theorem 5.5.4}), we have
\begin{align}
	&\left(\mathcal{Q}\left(A[g], A\left[f_2\right], \cdots, A\left[f_n\right]\right)\right)^2\nonumber
	\\
	 \geq &\mathcal{Q}\left(A\left[f_2\right], A\left[f_2\right], \cdots, A\left[f_n\right]\right) \cdot \mathcal{Q}\left(A[g], A[g], A\left[f_2\right], \cdots, A\left[f_n\right]\right),\label{equ:Q2>QQ}
\end{align}
together with Lemma \ref{lemma:V12=V21}, for any function $g \in \operatorname{dom}(\mathcal{A})$ we have
	\begin{align}
	\langle\mathcal{A} g, \mathcal{A} g\rangle_{L^2(\omega)} & =\frac{1}{n+1} \int_{\mathcal{C}_{\omega_0}} \frac{f_2\left(\mathcal{Q}\left(A[g], A\left[f_2\right], \cdots, A\left[f_n\right]\right)\right)^2}{\mathcal{Q}\left(A\left[f_2\right], A\left[f_2\right], \cdots, A\left[f_n\right]\right)} F(\nu)	\mathrm{d} {\mu_{g}} \nonumber
	\\
	& \geq \frac{1}{n+1} \int_{\mathcal{C}_{\omega_0}} f_2 \mathcal{Q}\left(A[g], A[g], A\left[f_2\right], \cdots, A\left[f_n\right]\right) F(\nu)	\mathrm{d} {\mu_{g}} \nonumber
	\\
	& =\frac{1}{n+1} \int_{\mathcal{C}_{\omega_0}} g \mathcal{Q}\left(A[g], A\left[f_2\right], \cdots, A\left[f_n\right]\right) F(\nu)	\mathrm{d} {\mu_{g}} \nonumber
	\\
	& =\langle g, \mathcal{A} g\rangle_{L^2(\omega)} .\label{equ:AgAg>gAg}
\end{align}
Therefore, if $\lambda$ is an eigenvalue of the operator $\mathcal{A}$, the above inequality implies that $\lambda^2 \geq \lambda$, which means $\lambda \geq 1$ or $\lambda \leq 0$. Consequently, the positive eigenspace of $\mathcal{A}$ is of one dimension and is spanned by $f_2$. In view of Lemma  \ref{lemma:3.1}, we complete the proof of \eqref{equ:AAf2>fAffAf}.

Next, we characterize the equality case. If equality holds in \eqref{equ:V2>VV}, then also equality holds in \eqref{equ:AAf2>fAffAf}, that is
\begin{align}	
\left(\left\langle f, \mathcal{A} f_1\right\rangle_{L^2(\omega)}\right)^2=\langle f, \mathcal{A} f\rangle_{L^2(\omega)}\left\langle f_1, \mathcal{A} f_1\right\rangle_{L^2(\omega)} .\label{equ:fAf2=}
\end{align}

From \cite{Shenfeld19}*{Lemma 2.9 and its proof}, equality holds in \eqref{equ:fAf2=} if and only if $\widetilde{f}:=f-a f_1 \in$ $\operatorname{Ker}(\mathcal{A})$ for some constant $a \in \mathbb{R}$.

For $g=\tilde{f}$, equality holds in \eqref{equ:AgAg>gAg} and in turn in Alexandrov's mixed discriminant inequality \eqref{equ:Q2>QQ}. It follows from \cite{book-convex-body}*{Theorem 5.5.4} that $A[\tilde{f}]=c A\left[f_2\right]$ for some function $c: \mathcal{C}_{\omega_0} \rightarrow \mathbb{R}$. Since $\tilde{f} \in \operatorname{Ker}(\mathcal{A})$, we have
$$
0=\frac{f_2 \mathcal{Q}\left(A\left[\tilde{f}-c f_2\right], A\left[f_2\right], \cdots, A\left[f_n\right]\right)}{\mathcal{Q}\left(A\left[f_2\right], A\left[f_2\right], \cdots, A\left[f_n\right]\right)}=\mathcal{A} \tilde{f}-cf_2=-cf_2,
$$
which implies $c=0$, since $f_2>0$. Then $A[\tilde{f}]=0$,  combining with the fact that $\tilde{f}$ is an anisotropic capillary function, by Lemma \ref{lemma:ker-A}, we obtain   $$\tilde{f}=\sum_{i=1}^n a_iG(\mathcal{T}^{-1}(\xi))\(\mathcal{T}^{-1}(\xi), E_i\)$$ for some constants $a_i \in \mathbb{R}, i=1, \cdots, n$. Therefore, equality in \eqref{equ:V2>VV} holds if and only if $f=a f_1+\sum_{i=1}^n a_iG(\mathcal{T}^{-1}(\xi))\(\mathcal{T}^{-1}(\xi), E_i\)$. Combining with \eqref{equ:G(vF,Y)=<v,Y>/F}, we complete the proof.
\end{proof}
 \begin{proof}[\textbf{Proof of Theoren \ref{thm:A-F-convex}}]
 	Let $f$ and $f_i(i=1, \cdots, n)$ be the anisotropic capillary support functions $\hat{s}_{k}\ (1\leq k\leq n+1)$ of $\widehat{\Sigma}_1$ and $\widehat{\Sigma}_{i+1}(i=1, \cdots, n)$, in Theorem \ref{thm:A-F-convex}.  Obviously, $f$ and $f_i$ are anisotropic capillary convex functions. By a horizontal translation, we may assume $f_i>0$. Applying Theorem \ref{thm:V2>VV} and taking into account of Proposition \ref{prop7.3}, we obtain the desired inequality. Equality holds if and only if
 	\begin{align}\label{equ:f=af+...}
 		f=a f_1+\sum_{i=1}^n a_iG(\mathcal{T}^{-1}(\xi))\(\mathcal{T}^{-1}(\xi), E_i\).
 	\end{align}
 	It follows from $A[f]>0$, \eqref{equ:A[X,v]} and Remark \ref{rk:3} that the constant $a>0$.  For any convex anisotropic $\omega_0$-capillary hypersurface, we have  $\nu_{{F}}(X(\xi))=\mathcal{T}^{-1}(\xi)$ for $\xi\in\mathcal{C}_{\omega_0}$ by \eqref{equ:X(xi)=t-1VF}, which yields \eqref{equ:s=as+...} together  with \eqref{equ:G(vF,Y)=<v,Y>/F} and \eqref{equ:f=af+...}.
 \end{proof}

An application of Theorem \ref{thm:A-F-convex}, combined with the methodology developed in \cite{book-convex-body}*{Section 7.4 and equation (7.63)}, yields the following inequality for anisotropic capillary convex bodies.
 \begin{corollary}\label{cor:4.3}
 	Let $m \in\{1, \cdots, n+1\}$ and assume $K_0, K_1, \cdots, K_{m+1}, \cdots, K_{n+1} \in \mathcal{K}_{\omega_0}$, and $\hat{s}_j (0\le j\le n+1)$ are the capillary support function of $K_j$. Let
$$
V_{(i), {\omega_0}}:=V(\underbrace{\hat{s}_0, \cdots, \hat{s}_0}_{(m-i) \text { copies }} \underbrace{\hat{s}_1, \cdots, \hat{s}_1}_{i \text { copies }}, \hat{s}_{m+1}, \cdots, \hat{s}_{n+1}), \quad \text { for } \quad i=0, \cdots, m .
$$
Then $V_{(i), {\omega_0}}$ satisfies
$$
V_{(j), {\omega_0}}^{k-i} \geq V_{(i), {\omega_0}}^{k-j} V_{(k), {\omega_0}}^{j-i}, \quad 0 \leq i<j<k \leq m \leq n+1 .
$$
Equality holds if and only if $\hat{s}_0=a \hat{s}_1+\sum_{i=1}^n a_i\frac{\<\nu, E_i\>}{F(\nu)}$ for some constants $a>0$ and $a_i, 1 \leq i \leq n$.
 \end{corollary}
As a consequence, we finally arrive Theorem \ref{thm:A-F-k}.
\begin{proof}[\textbf{Proof of Theorem \ref{thm:A-F-k}}] By choosing $k=m=n+1, K_0=\widehat{\Sigma}$ and $K_1=\widehat{\mathcal{C}}_{\omega_0}$ in Corollary \ref{cor:4.3}, together with the fact that $V_{(n+1), {\omega_0}}=\left|\widehat{\mathcal{C}}_{\omega_0}\right|$ and Lemma \ref{lemma:V-k=V(KKK,LLL)}, we obtain the desired result.
\end{proof}

\section*{Reference}


\end{document}